\documentclass[oneside,english]{amsart}
\usepackage[T1]{fontenc}
\usepackage[latin9]{inputenc}
\usepackage{amsthm}
\usepackage{amstext}

\makeatletter

\providecommand{\tabularnewline}{\\}

\numberwithin{equation}{section}
\numberwithin{figure}{section}
\theoremstyle{plain}
\newtheorem{thm}{\protect\theoremname}
\ifx\proof\undefined
\newenvironment{proof}[1][\protect\proofname]{\par
\normalfont\topsep6\p@\@plus6\p@\relax
\trivlist
\itemindent\parindent
\item[\hskip\labelsep\scshape #1]\ignorespaces
}{%
\endtrivlist\@endpefalse
}
\providecommand{\proofname}{Proof}
\fi

\makeatother

\usepackage{babel}
\providecommand{\theoremname}{Theorem}

\begin{document}

\title{\noindent General solutions of sums of consecutive cubed integers
equal to squared integers }

\author{\noindent Vladimir Pletser}

\address{\noindent European Space Research and Technology Centre, ESA-ESTEC
P.O. Box 299, NL-2200 AG Noordwijk, The Netherlands E-mail: Vladimir.Pletser@esa.int }
\begin{abstract}
All integer solutions $\left(M,a,c\right)$ to the problem of the
sums of $M$ consecutive cubed integers $\left(a+i\right)^{3}$ ($a>1$,
$0\leq i\leq M-1$) equaling squared integers $c^{2}$ are found by
decomposing the product of the difference and sum of the triangular
numbers of $\left(a+M-1\right)$ and $\left(a-1\right)$ in the product
of their greatest common divisor $g$ and remaining square factors
$\delta^{2}$ and $\sigma^{2}$, yielding $c=g\delta\sigma$. Further,
the condition that $g$ must be integer for several particular and
general cases yield generalized Pell equations whose solutions allow
to find all integer solutions $\left(M,a,c\right)$ showing that these
solutions appear recurrently. In particular, it is found that there
always exist at least one solution for the cases of all odd values
of $M$, of all odd integer square values of $a$, and of all even
values of $M$ equal to twice an integer square.

\end{abstract}

\maketitle

\section{\noindent Introduction}

\noindent It is known since long that the sum of $M$ consecutive
cubed positive integers starting from $1$ equals the square of the
sum of the $M$ consecutive integers, which itself equals the triangular
number $\triangle_{M}$ of the number of terms $M$, 
\begin{equation}
\sum_{i=1}^{M}i^{3}=\left(\sum_{i=1}^{M}i\right)^{2}=\left(\frac{M\left(M+1\right)}{2}\right)^{2}=\triangle_{M}^{2}\label{eq:1-2}
\end{equation}

\noindent for $\forall M\in\mathbb{Z}^{+}$. The question whether
this remarkable result can be extended to other integer values of
the starting point, i.e. whether the sum of consecutive cubed positive
integers starting from $a\neq1$ is also a perfect square 
\begin{equation}
\sum_{i=0}^{M-1}\left(a+i\right)^{3}=c^{2}\label{eq:2-3}
\end{equation}
has been addressed by several authors but has received so far only
partial answers. 

\noindent With the notation of this paper, Lucas stated \cite{key-1-40}
that the only solutions for $M=5$ are $a=0,1,96$ and $118$ (missing
the solution $a=25$, see further Table 1), and that there are no
other solutions for $M=2$ than $a=1$. Aubry showed \cite{key-1-41}
that a solution for $M=3$ is $a=23$, $c=204$, correcting Lucas'
statement that there are no other solution for $M=3$ than $a=1$.
Other historical accounts can be found in \cite{key-1-13}. Cassels
proved \cite{key-1-43} by using the method of finding all integral
points on a given curve of genus $1$ $y^{2}=3x\left(x^{2}+2\right)$
(with $x=a+1$, $y=c$ in this paper notations), that the only solutions
for $M=3$ are $a=-1,0,1$ and $23$. Stroeker obtained \cite{key-1-44}
complete solutions for $2\leq M\leq50$ and $M=98$, using estimates
of lower bound of linear forms in elliptic logarithms to solve elliptic
curve equations of the form $Y^{2}=X^{3}+dX$ where $d=n^{2}\left(n^{2}-1\right)/4$,
$X=nx+n\left(n-1\right)/2$, $Y=ny$ (with $n=M$, $x=a$, $y=c$
in this paper notations). The method reported, although powerful,
appears long and difficult and caused some problems for the cases
$M=41$ and $44$. Stroeker remarked also that $M=a=33$ with $c=2079=33\times63$.
This is not the single occurrence of $M=a$, as it occurs also for
$M=a=2017,124993,7747521,...$ (see \cite{key-5-2,key-5-3}).

\noindent One of the reasons that these previous attempts to find
all solutions $\left(M,a,c\right)$ to (\ref{eq:2-3}) were only partially
successful was most likely due to the approach taken to start the
search for solutions for single values of $M$, one by one and in
an increasing order of $M$ values. The method proposed in this paper
is to tackle the problem in a different way and instead of looking
at each individual values of $M$ one by one, to consider the problem
in a more global approach by comparing different sets of known solutions,
and instead of listing solutions in increasing order of $M$ values,
to look at two other parameters, $\delta$ and $\sigma$, defined
further. This new beginning then leads to a more classical approach
using general solutions of Pell equations, that allows to find all
solutions in $\left(M,a,c\right)$ of (\ref{eq:2-3}) for all possible
cases. Note that Pell equations were already used previously by various
authors (e.g. Catalan \cite{key-1-45,key-1-46}, Cantor \cite{key-1-47},
Richaud \cite{key-1-48}) in the 19th century in attempts to solve
the present problem, albeit without reaching a complete resolution
of the problem. 

\noindent The approach proposed in this paper includes three steps.

\noindent Step 1 in Section 2 is based on the decomposition of the
product of the difference $\Delta$ and the sum $\Sigma$ of the two
triangular numbers of $\left(a+M-1\right)$ and $\left(a-1\right)$
in simple factors, $g,\delta,\sigma$ where $g=\gcd\left(\Delta,\Sigma\right)$,
allowing to find general expressions of $a$ and $c$ in function
of $M,\delta$ and $\sigma$ that are always solutions of (\ref{eq:2-3}).

\noindent In step 2 in Section 3, some conditions on $\delta$ and
$\sigma$ are explored to obtain three particular cases of solutions
yielding specific expressions of $M$ in function of $k,$ $\forall k\in\mathbb{Z}^{+}$,
including the case of $M$ taking all odd positive integer values.
Section 4 recalls some basics on Pell equation solutions to introduce
the third step.

\noindent In step 3 in Section 5, a general solution in $\left(M,a,c\right)$
is found for all values of $\delta$ and $\sigma$, based on solutions
of simple and generalized Pell equations involving Chebyshev polynomials,
allowing to find all the solutions in $\left(M,a,c\right)$ to (\ref{eq:2-3}).
As an alternative to step 3, recurrence relations are deduced in Section
6. Section 7 summarizes all findings.

\noindent Trivial solutions are not considered here. For instance,
$M\leq0$ is meaningless; for $M=1$, the only solution is obviously
$a=\alpha^{2}$, yielding $\left(M,a,c\right)=\left(1,\alpha^{2},\alpha^{3}\right)$;
therefore, we consider only $M>1$. If $a<0$, there are no solutions
if $M<\left(1-2a\right)$; if $M=\left(1-2a\right)$, the only solution
is $\left(M,a,c\right)=\left(\left(1-2a\right),a,0\right)$; if $M>1-2a$,
the solutions are $\left(M,a,c\right)=\left(M,\left(1-a\right),\left(\triangle_{M+a-1}^{2}-\triangle_{\left(-a\right)}^{2}\right)\right)$,
i.e. a shift of the first term from a negative value $a$ to a positive
value $\left(1-a\right)$ with a reduction of the number of terms
$\left(M+2a-1\right)$. If $a=0$ or $1$, the classical solutions
(\ref{eq:1-2}) are $\left(M,a,c\right)=\left(M,0,\triangle_{M-1}^{2}\right)$
or $\left(M,1,\triangle_{M}^{2}\right)$. Therefore, we limit our
search to solutions for $M>1$ and $a>1$.

\section{Step 1: A General theorem}
\begin{thm}
For $\forall\delta\in\mathbb{Z}^{+}$, $\exists\sigma,a,M,c\in\mathbb{Z}^{+}$,
$\kappa\in\mathbb{Q}^{+}$, such as $\delta<\sigma$, $\gcd\left(\delta,\sigma\right)=1$,
$\kappa=\left(\sigma/\delta\right)>1$, 
\begin{equation}
\sum_{i=0}^{M-1}\left(a+i\right)^{3}=c^{2}\label{eq:4-1}
\end{equation}

\noindent holds if
\begin{eqnarray}
a & = & \frac{M\left(\kappa^{2}-1\right)+1+\sqrt{M^{2}\left(\kappa^{4}-1\right)+1}}{2}\label{eq:2-2}\\
c & = & \frac{\kappa M}{2}\left(\kappa^{2}M+\sqrt{M^{2}\left(\kappa^{4}-1\right)+1}\right)\label{eq:3-2}
\end{eqnarray}
\end{thm}
\begin{proof}
\noindent For $a,M,c,\Delta,\varSigma,g,\Delta^{\prime},\varSigma^{\prime},\delta,\sigma,C\in\mathbb{Z}^{+}$,
$i,k\in\mathbb{Z}^{*}$, $\kappa\in\mathbb{\mathbb{Q}}^{+}$, with
$M>1$, $0\leq i\leq M-1$, $\kappa=\left(\sigma/\delta\right)>1$,
the sum of cubes of $M$ consecutive integers $\left(a+i\right)$
for $i=0$ to $M-1$ can be written successively as
\begin{eqnarray}
\sum_{i=0}^{M-1}\left(a+i\right)^{3} & = & \sum_{i=0}^{a+M-1}i^{3}-\sum_{i=0}^{a-1}i^{3}\label{eq:1-1}\\
 & = & \triangle_{a+M-1}^{2}-\triangle_{a-1}^{2}\label{eq:2-1}\\
 & = & \left(\triangle_{a+M-1}-\triangle_{a-1}\right)\left(\triangle_{a+M-1}+\triangle_{a-1}\right)\\
 & = & \left(\Delta\right)\left(\Sigma\right)\label{eq:3-1}
\end{eqnarray}
where $\Delta=\triangle_{a+M-1}-\triangle_{a-1}$ and $\Sigma=\triangle_{a+M-1}+\triangle_{a-1}$,
i.e. the difference and the sum of the triangular numbers of $\left(a+M-1\right)$
and $\left(a-1\right)$, with obviously $\Delta<\Sigma$, that can
also be written
\begin{eqnarray}
\Delta & = & M\left(a+\frac{M-1}{2}\right)\label{eq:5-3}\\
\Sigma & = & a^{2}+a\left(M-1\right)+\frac{M\left(M-1\right)}{2}\label{eq:7-2}
\end{eqnarray}
Let $g=\gcd\left(\Delta,\Sigma\right)$, yielding $\Delta=g\Delta^{\prime}$
and $\Sigma=g\Sigma^{\prime}$. For (\ref{eq:4-1}) to hold, $c^{2}=g^{2}\Delta^{\prime}\Sigma^{\prime}$
and as $\Delta^{\prime}$ and $\Sigma^{\prime}$ are coprimes and
their product must be square, both must be integer squares, i.e. $\Delta^{\prime}=\delta^{2}$
and $\Sigma^{\prime}=\sigma^{2}$, with $\gcd\left(\delta,\sigma\right)=1$
and $\delta<\sigma$, yielding 
\begin{equation}
c=g\delta\sigma\label{eq:12-3}
\end{equation}

\noindent From (\ref{eq:5-3}) and (\ref{eq:7-2}), one has then respectively
\begin{eqnarray}
2a+M-1 & = & \frac{2g\delta^{2}}{M}\label{eq:12-1}\\
 & = & \sqrt{4g\sigma^{2}-\left(M^{2}-1\right)}\label{eq:13-2}
\end{eqnarray}

\noindent where the $+$ sign is taken in front of the square root
in (\ref{eq:13-2}) as $2a+M>1$. Solving for $g$ yields then
\begin{equation}
g=\frac{M}{2\delta^{2}}\left(\kappa^{2}M+\sqrt{M^{2}\left(\kappa^{4}-1\right)+1}\right)\label{eq:14-1}
\end{equation}

\noindent Replacing in (\ref{eq:12-1}) or (\ref{eq:13-2}) yields
then directly (\ref{eq:2-2}) and in (\ref{eq:12-3}) yields directly
(\ref{eq:3-2}). 
\end{proof}

\section{Step 2: Three Particular Solutions and a Generalization}

\noindent For which values of $M$ does (\ref{eq:4-1}) hold ? Answers
can be found for at least three particular cases. Other general approaches
are developed further.

\noindent Table 1 gives for $1<M\leq45$ and $1<a<10^{5}$, the first
values of $M,a$ and associated $\Delta,\Sigma,g,\delta,\sigma$ and
$c$ values. For $M$ and $a<10^{5}$, there are 892 solutions $\left(M,a,c\right)$,
given in \cite{key-5-5,key-6-0}, such that (\ref{eq:1}) holds.
\begin{table}
\protect\caption{First values of $M,a,\Delta,\Sigma,g,\delta,\sigma$ and $c$ for
$1<M\leq45$ and $1<a<10^{5}$ }

\centering{}%
\begin{tabular}{|c|c|c|c|c|c|c|c|}
\hline 
$M$ & $a$ & $\Delta$ & $\Sigma$ & $g$ & $\delta$ & $\sigma$ & $c$\tabularnewline
\hline 
\hline 
\textbf{\small{}3} & \textbf{\small{}23} & \textbf{\small{}72} & \textbf{\small{}578} & \textbf{\small{}2} & \textbf{\small{}6} & \textbf{\small{}17} & \textbf{\small{}204}\tabularnewline
\hline 
{\small{}5} & {\small{}25} & {\small{}135} & {\small{}735} & {\small{}15} & {\small{}3} & {\small{}7} & {\small{}315}\tabularnewline
\hline 
{\small{}5} & {\small{}96} & {\small{}490} & {\small{}9610} & {\small{}10} & {\small{}7} & {\small{}31} & {\small{}2170}\tabularnewline
\hline 
\textbf{\small{}5} & \textbf{\small{}118} & \textbf{\small{}600} & \textbf{\small{}14406} & \textbf{\small{}6} & \textbf{\small{}10} & \textbf{\small{}49} & \textbf{\small{}2940}\tabularnewline
\hline 
\textbf{\small{}7} & \textbf{\small{}333} & \textbf{\small{}2352} & \textbf{\small{}112908} & \textbf{\small{}12} & \textbf{\small{}14} & \textbf{\small{}97} & \textbf{\small{}16296}\tabularnewline
\hline 
\emph{\small{}8} & \emph{\small{}28} & \emph{\small{}252} & \emph{\small{}1008} & \emph{\small{}252} & \emph{\small{}1} & \emph{\small{}2} & \emph{\small{}504}\tabularnewline
\hline 
\textbf{\small{}9} & \textbf{\small{}716} & \textbf{\small{}6480} & \textbf{\small{}518420} & \textbf{\small{}20} & \textbf{\small{}18} & \textbf{\small{}161} & \textbf{\small{}57960}\tabularnewline
\hline 
\textbf{\small{}11} & \textbf{\small{}1315} & \textbf{\small{}14520} & \textbf{\small{}1742430} & \textbf{\small{}30} & \textbf{\small{}22} & \textbf{\small{}241} & \textbf{\small{}159060}\tabularnewline
\hline 
{\small{}12} & {\small{}14} & {\small{}234} & {\small{}416} & {\small{}26} & {\small{}3} & {\small{}4} & {\small{}312}\tabularnewline
\hline 
{\small{}13} & {\small{}144} & {\small{}1950} & {\small{}22542} & {\small{}78} & {\small{}5} & {\small{}17} & {\small{}6630}\tabularnewline
\hline 
\textbf{\small{}13} & \textbf{\small{}2178} & \textbf{\small{}28392} & \textbf{\small{}4769898} & \textbf{\small{}42} & \textbf{\small{}26} & \textbf{\small{}337} & \textbf{\small{}368004}\tabularnewline
\hline 
{\small{}15} & {\small{}25} & {\small{}480} & {\small{}1080} & {\small{}120} & {\small{}2} & {\small{}3} & {\small{}720}\tabularnewline
\hline 
\textbf{\small{}15} & \textbf{\small{}3353} & \textbf{\small{}50400} & \textbf{\small{}11289656} & \textbf{\small{}56} & \textbf{\small{}30} & \textbf{\small{}449} & \textbf{\small{}754320}\tabularnewline
\hline 
{\small{}15} & {\small{}57960} & {\small{}869505} & {\small{}3360173145} & {\small{}105} & {\small{}91} & {\small{}5657} & {\small{}54052635}\tabularnewline
\hline 
{\small{}17} & {\small{}9} & {\small{}289} & {\small{}361} & {\small{}1} & {\small{}17} & {\small{}19} & {\small{}323}\tabularnewline
\hline 
{\small{}17} & {\small{}120} & {\small{}2176} & {\small{}16456} & {\small{}136} & {\small{}4} & {\small{}11} & {\small{}5984}\tabularnewline
\hline 
\textbf{\small{}17} & \textbf{\small{}4888} & \textbf{\small{}83232} & \textbf{\small{}23970888} & \textbf{\small{}72} & \textbf{\small{}34} & \textbf{\small{}577} & \textbf{\small{}1412496}\tabularnewline
\hline 
\emph{\small{}18} & \emph{\small{}153} & \emph{\small{}2907} & \emph{\small{}26163} & \emph{\small{}2907} & \emph{\small{}1} & \emph{\small{}3} & \emph{\small{}8721}\tabularnewline
\hline 
{\small{}18} & {\small{}680} & {\small{}12393} & {\small{}474113} & {\small{}17} & {\small{}27} & {\small{}167} & {\small{}76653}\tabularnewline
\hline 
\textbf{\small{}19} & \textbf{\small{}6831} & \textbf{\small{}129960} & \textbf{\small{}46785690} & \textbf{\small{}90} & \textbf{\small{}38} & \textbf{\small{}721} & \textbf{\small{}2465820}\tabularnewline
\hline 
{\small{}21} & {\small{}14} & {\small{}504} & {\small{}686} & {\small{}14} & {\small{}6} & {\small{}7} & {\small{}588}\tabularnewline
\hline 
{\small{}21} & {\small{}144} & {\small{}3234} & {\small{}23826} & {\small{}66} & {\small{}7} & {\small{}19} & {\small{}8778}\tabularnewline
\hline 
\textbf{\small{}21} & \textbf{\small{}9230} & \textbf{\small{}194040} & \textbf{\small{}85377710} & \textbf{\small{}110} & \textbf{\small{}42} & \textbf{\small{}881} & \textbf{\small{}4070220}\tabularnewline
\hline 
\textbf{\small{}23} & \textbf{\small{}12133} & \textbf{\small{}279312} & \textbf{\small{}147476868} & \textbf{\small{}132} & \textbf{\small{}46} & \textbf{\small{}1057} & \textbf{\small{}6418104}\tabularnewline
\hline 
\textbf{\small{}25} & \textbf{\small{}15588} & \textbf{\small{}390000} & \textbf{\small{}243360156} & \textbf{\small{}156} & \textbf{\small{}50} & \textbf{\small{}1249} & \textbf{\small{}9742200}\tabularnewline
\hline 
\textbf{\small{}27} & \textbf{\small{}19643} & \textbf{\small{}530712} & \textbf{\small{}386358518} & \textbf{\small{}182} & \textbf{\small{}54} & \textbf{\small{}1457} & \textbf{\small{}14319396}\tabularnewline
\hline 
{\small{}28} & {\small{}81} & {\small{}2646} & {\small{}9126} & {\small{}54} & {\small{}7} & {\small{}13} & {\small{}4914}\tabularnewline
\hline 
\textbf{\small{}29} & \textbf{\small{}24346} & \textbf{\small{}706440} & \textbf{\small{}593409810} & \textbf{\small{}210} & \textbf{\small{}58} & \textbf{\small{}1681} & \textbf{\small{}20474580}\tabularnewline
\hline 
\textbf{\small{}31} & \textbf{\small{}29745} & \textbf{\small{}922560} & \textbf{\small{}885657840} & \textbf{\small{}240} & \textbf{\small{}62} & \textbf{\small{}1921} & \textbf{\small{}28584480}\tabularnewline
\hline 
{\small{}32} & {\small{}69} & {\small{}2704} & {\small{}7396} & {\small{}4} & {\small{}26} & {\small{}43} & {\small{}4472}\tabularnewline
\hline 
{\small{}32} & {\small{}133} & {\small{}4752} & {\small{}22308} & {\small{}132} & {\small{}6} & {\small{}13} & {\small{}10296}\tabularnewline
\hline 
\emph{\small{}32} & \emph{\small{}496} & \emph{\small{}16368} & \emph{\small{}261888} & \emph{\small{}16368} & \emph{\small{}1} & \emph{\small{}4} & \emph{\small{}65472}\tabularnewline
\hline 
{\small{}33} & {\small{}33} & {\small{}1617} & {\small{}2673} & {\small{}33} & {\small{}7} & {\small{}9} & {\small{}2079}\tabularnewline
\hline 
\textbf{\small{}33} & \textbf{\small{}35888} & \textbf{\small{}1184832} & \textbf{\small{}1289097488} & \textbf{\small{}272} & \textbf{\small{}66} & \textbf{\small{}2177} & \textbf{\small{}39081504}\tabularnewline
\hline 
{\small{}35} & {\small{}225} & {\small{}8470} & {\small{}58870} & {\small{}70} & {\small{}11} & {\small{}29} & {\small{}22330}\tabularnewline
\hline 
\textbf{\small{}35} & \textbf{\small{}42823} & \textbf{\small{}1499400} & \textbf{\small{}1835265906} & \textbf{\small{}306} & \textbf{\small{}70} & \textbf{\small{}2449} & \textbf{\small{}52457580}\tabularnewline
\hline 
\textbf{\small{}37} & \textbf{\small{}50598} & \textbf{\small{}1872792} & \textbf{\small{}2561979798} & \textbf{\small{}342} & \textbf{\small{}74} & \textbf{\small{}2737} & \textbf{\small{}69267996}\tabularnewline
\hline 
{\small{}39} & {\small{}111} & {\small{}5070} & {\small{}17280} & {\small{}30} & {\small{}13} & {\small{}24} & {\small{}9360}\tabularnewline
\hline 
\textbf{\small{}39} & \textbf{\small{}59261} & \textbf{\small{}2311920} & \textbf{\small{}3514118780} & \textbf{\small{}380} & \textbf{\small{}78} & \textbf{\small{}3041} & \textbf{\small{}90135240}\tabularnewline
\hline 
{\small{}40} & {\small{}3276} & {\small{}131820} & {\small{}10860720} & {\small{}780} & {\small{}13} & {\small{}118} & {\small{}1196520}\tabularnewline
\hline 
\textbf{\small{}41} & \textbf{\small{}68860} & \textbf{\small{}2824080} & \textbf{\small{}4744454820} & \textbf{\small{}420} & \textbf{\small{}82} & \textbf{\small{}3361} & \textbf{\small{}115752840}\tabularnewline
\hline 
{\small{}42} & {\small{}64} & {\small{}3549} & {\small{}7581} & {\small{}21} & {\small{}13} & {\small{}19} & {\small{}5187}\tabularnewline
\hline 
\textbf{\small{}43} & \textbf{\small{}79443} & \textbf{\small{}3416952} & \textbf{\small{}6314527758} & \textbf{\small{}462} & \textbf{\small{}86} & \textbf{\small{}3697} & \textbf{\small{}146889204}\tabularnewline
\hline 
{\small{}45} & {\small{}176} & {\small{}8910} & {\small{}39710} & {\small{}110} & {\small{}9} & {\small{}19} & {\small{}18810}\tabularnewline
\hline 
\textbf{\small{}45} & \textbf{\small{}91058} & \textbf{\small{}4098600} & \textbf{\small{}8295566906} & \textbf{\small{}506} & \textbf{\small{}90} & \textbf{\small{}4049} & \textbf{\small{}184391460}\tabularnewline
\hline 
\end{tabular}
\end{table}
 One observes very easily that:

\noindent (i) all odd values of $M$ have at least one entry (in bold
in Table 1) with $g=2\triangle_{\left(M-1\right)/2}=\left(M^{2}-1\right)/4$,
$\delta=2M$ and $\sigma=\left(2M^{2}-1\right)$, yielding $c=M\left(M^{2}-1\right)\left(2M^{2}-1\right)/2$
with $a=M^{3}-\left(3M-1\right)/2$;

\noindent (ii) those odd values of $a$ equal to odd integer squares
have at least one entry (e.g. for $M=17$ in Table 1) with $g=\left(a-1\right)/8$,
$\delta=\left(2a-1\right)$ and $\sigma=\left(2a+1\right)$, yielding
$c=\left(a-1\right)\left(4a^{2}-1\right)/8$ with $M=\left(\sqrt{a}-1\right)\left(2a-1\right)/2$

\noindent (iii) those even values of $M$ equal to twice an integer
square have at least one entry (in italics in Table 1) with $g=M\left(M^{2}-1\right)/2$,
$\delta=1$, $\sigma=\sqrt{M/2}$, yielding $c=M\left(M^{2}-1\right)\sqrt{M/2}/2$
with $a=\triangle_{M-1}$.

\noindent These three cases can be generalized respectively to all
odd values of $M$, to all odd integer square values of $a=\left(2k+1\right)^{2}$,
and to all even values of $M$ equal to twice integer squares in the
following 
\begin{thm}
\begin{flushleft}
$\forall k\in\mathbb{Z}^{+}$, $\exists\delta,\sigma,M,a,c\in\mathbb{Z}^{+}$
such that (\ref{eq:4-1}) holds :
\par\end{flushleft}

\noindent (i) if $\sigma=\left(\delta^{2}-2\right)/2$, with
\begin{eqnarray}
M & = & \left(2k+1\right)\label{eq:5-2}\\
a & = & \left(2k+1\right)^{3}-\left(3k+1\right)\label{eq:5-4}\\
 & = & M^{3}-\frac{\left(3M-1\right)}{2}\label{eq:5-5}\\
c & = & 2k\left(k+1\right)\left(2k+1\right)\left(8k\left(k+1\right)+1\right)\label{eq:6-1}\\
 & = & \frac{M\left(M^{2}-1\right)\left(2M^{2}-1\right)}{2}\label{eq:6-3}
\end{eqnarray}

\noindent (ii) if $\sigma=\delta+2$, with
\begin{eqnarray}
a & = & \left(2k+1\right)^{2}\label{eq:5-2-1}\\
M & = & k\left(8k\left(k+1\right)+1\right)\label{eq:6-4}\\
 & = & \frac{\left(\sqrt{a}-1\right)\left(2a-1\right)}{2}\label{eq:6-5}\\
c & = & \frac{k\left(k+1\right)}{2}\left(4\left(2k+1\right)^{4}-1\right)\label{eq:6-1-1}\\
 & = & \frac{\left(a-1\right)\left(4a^{2}-1\right)}{8}\label{eq:6-6}
\end{eqnarray}

\noindent (iii) if $\delta=1$ and $\sigma=k$ with $k>1$ and
\begin{eqnarray}
M & = & 2k^{2}\label{eq:5-2-2}\\
a & = & k^{2}\left(2k^{2}-1\right)\label{eq:6-7}\\
 & = & \frac{M\left(M-1\right)}{2}\label{eq:6-8}\\
c & = & k^{3}\left(4k^{4}-1\right)\label{eq:6-1-2}\\
 & = & \frac{M\left(M^{2}-1\right)}{2}\,\sqrt{\frac{M}{2}}\label{eq:6-9}
\end{eqnarray}
\end{thm}
\begin{proof}
For $k,\delta,\sigma,g,a,M>1,c\in\mathbb{Z}^{+}$, $q\in\mathbb{Q}^{+}$,
for (\ref{eq:4-1}) to hold:

\noindent (i) let $\sigma=\left(\delta^{2}-2\right)/2$. Replacing
in (\ref{eq:14-1}) yields
\begin{equation}
g=\frac{M^{2}}{8\delta^{4}}\left(\delta^{4}-4\delta^{2}+4+\sqrt{\delta^{8}-8\delta^{6}+8\delta^{4}+16\delta^{2}\left(\frac{\delta^{2}}{M^{2}}-2\right)+16}\right)\label{eq:30-2}
\end{equation}

\noindent For the polynomial in $\delta$ under the square root sign
to be a square, let $\left(\delta/M\right)^{2}-2=2$, i.e. $\delta=2M$,
yielding $\delta^{8}-8\delta^{6}+8\delta^{4}+32\delta^{2}+16=\left(\delta^{4}-4\delta^{2}-4\right)^{2}$,
giving $g=M^{2}\left(\delta^{2}-4\right)/4\delta^{2}=\left(M^{2}-1\right)/4$
where $\delta$ was replaced by $2M$. As $g\in\mathbb{Z}^{+}$, $M$
cannot be even and must be odd, i.e. $M=2k+1$ $\forall k\in\mathbb{Z}^{+}$,
yielding $g=k\left(k+1\right)$, $\delta=2M=2\left(2k+1\right)$,
$\sigma=2M^{2}-1=8k\left(k+1\right)+1$. Replacing in (\ref{eq:2-2})
and (\ref{eq:3-2}) yields directly (\ref{eq:5-4}) to (\ref{eq:6-3}). 

\noindent (ii) Let $\sigma=\delta+2$. Replacing in (\ref{eq:14-1})
yields
\begin{equation}
g=\frac{M}{2\delta^{4}}\left(M\left(\delta+2\right)^{2}+\sqrt{M^{2}\left(\left(\delta+2\right)^{4}-\delta^{4}\right)+\delta^{4}}\right)\label{eq:31-2}
\end{equation}

\noindent For the polynomial in $\delta$ under the square root sign
to be a square, let $M=q\delta$ and $\delta=8q(q+1)+1$ with $q\in\mathbb{Q}^{+}$.
It yields then $\sqrt{q^{2}\left(\left(\delta+2\right)^{4}-\delta^{4}\right)+\delta^{2}}=\delta^{2}-4q\left(\delta+1\right)$,
giving $g=q\left(q+1\right)/2$. As $g\in\mathbb{Z}^{+}$, $q\in\mathbb{Z}^{+}$
and let $q=k$, yielding $g=\triangle_{k}$, $\delta=16\triangle_{k}+1$,
$\sigma=16\triangle_{k}+3$, $M=k\delta=k\left(16\triangle_{k}+1\right)=$(\ref{eq:6-4}).
Replacing in (\ref{eq:12-1}) and (\ref{eq:12-3}) yields (\ref{eq:5-2-1})
and (\ref{eq:6-1-1}). Replacing $k$ in function of $a$ from (\ref{eq:5-2-1})
yields also (\ref{eq:6-5}) and (\ref{eq:6-6}).

\noindent (iii) Let $\delta=1$ and $\sigma=k$ with $k>1$. Then
(\ref{eq:14-1}) reads 
\begin{equation}
g=\frac{M\left(k^{2}M+\sqrt{M^{2}\left(k^{4}-1\right)+1}\right)}{2}\label{eq:31}
\end{equation}
which takes integer values if $M=2k^{2}$, yielding $g=k^{2}\left(4k^{4}-1\right)=M\left(M^{2}-1\right)/2$.
Replacing in (\ref{eq:12-1}) and (\ref{eq:12-3}) yields directly
(\ref{eq:6-7}) to (\ref{eq:6-9}).
\end{proof}
\noindent The case (i) of Theorem 2 confirms the statement of Stroeker
(\cite{key-1-44}, p. 297) about all odd values of $M$ (in bold in
Table 1) having a solution to (\ref{eq:4-1}). The first $50\,000$
values of $\left(M,a,c\right)$ for this case (i) are given in \cite{key-7-1,key-7-2}

\noindent For the case (ii) of Theorem 2, Table 2 gives the first
five values of $M,a,g,\delta,\sigma$ and $c$.
\begin{table}
\protect\caption{Values of $M,a,g,\delta,\sigma,c$ for $1\leq k\leq5$ for case (ii)
of Theorem 2}

\centering{}%
\begin{tabular}{|c|c|c|c|c|c|c|}
\hline 
$k$ & $M$ & $a$ & $g$ & $\delta$ & $\sigma$ & $c$\tabularnewline
\hline 
\hline 
1 & 17 & 9 & 1 & 17 & 19 & 323\tabularnewline
\hline 
2 & 98 & 25 & 3 & 49 & 51 & 7497\tabularnewline
\hline 
3 & 291 & 49 & 6 & 97 & 99 & 57618\tabularnewline
\hline 
4 & 644 & 81 & 10 & 161 & 163 & 262430\tabularnewline
\hline 
5 & 1205 & 121 & 15 & 241 & 241 & 878445\tabularnewline
\hline 
\end{tabular}
\end{table}
The first $50\,000$ values of $\left(M,a,c\right)$ for this case
(ii) are given in \cite{key-7-3,key-7-4}.

\noindent As there exist other triplets of values of $\left(M,a,c\right)$
such that (\ref{eq:4-1}) holds with $M=k\delta$ for the same value
of $g=\Delta_{k}$, the case (ii) of Theorem 2 can be generalized
to other values of $a$ as follows. Table 3 shows some of these values
for $1\leq k\leq2$ and $1<M_{n},a_{n}<10^{5}$ and indexed by $n$
for increasing values of $M_{n}$.
\begin{table}
\protect\caption{Values of $M_{n},a_{n},g,\delta_{n},\sigma_{n},c_{n}$ for $1\leq k\leq2$
and $1<M_{n},a_{n}<10^{5}$}

\centering{}%
\begin{tabular}{|c|c|c|c|c|c|c|c|}
\hline 
$k$ & $n$ & $M_{n}$ & $a_{n}$ & $g$ & $\delta_{n}$ & $\sigma_{n}$ & $c_{n}$\tabularnewline
\hline 
\hline 
1 & 1 & 17 & 9 & 1 & 17 & 19 & 323\tabularnewline
\hline 
1 & 2 & 305 & 153 & 1 & 305 & 341 & 104005\tabularnewline
\hline 
1 & 3 & 5473 & 2737 & 1 & 5473 & 6119 & 33489287\tabularnewline
\hline 
1 & 4 & 98209 & 49105 & 1 & 98209 & 109801 & 10783446409\tabularnewline
\hline 
\hline 
2 & 1 & 98 & 25 & 3 & 49 & 51 & 7497\tabularnewline
\hline 
2 & 2 & 4898 & 1225 & 3 & 2449 & 2549 & 18727503\tabularnewline
\hline 
\end{tabular}
\end{table}

\noindent It is easily seen from Table 3 that, for each value of $k$,
\begin{eqnarray}
\sigma_{n}-\delta_{n} & = & \sigma_{n-1}+\delta_{n-1}\label{eq:32-2}\\
\sigma_{n}+\delta_{n} & = & \left(2\left(2k+1\right)^{2}-1\right)\sigma_{n-1}+\left(2\left(2k+1\right)^{2}+1\right)\delta_{n-1}\label{eq:33-2}
\end{eqnarray}

\noindent with $\sigma_{0}=\delta_{0}=1$, yielding the recurrence
relations
\begin{eqnarray}
\delta_{n} & = & 2\left(2k+1\right)^{2}\delta_{n-1}-\delta_{n-2}\label{eq:34-1}\\
\sigma_{n} & = & 2\left(2k+1\right)^{2}\sigma_{n-1}-\sigma_{n-2}\label{eq:35-1}
\end{eqnarray}

\noindent i.e. $\delta_{n}$ and $\sigma_{n}$ fulfill the Diophantine
equation $\left(2k\left(k+1\right)+1\right)\delta_{n}^{2}-2k\left(k+1\right)\sigma_{n}^{2}=1$.
The general solutions of this Diophantine equation can be expressed
in function of Chebyshev polynomials of the second kind $U_{n}\left(\left(2k+1\right)^{2}\right)$
as 
\begin{eqnarray}
\delta_{n} & = & U_{n}\left(\left(2k+1\right)^{2}\right)-U_{n-1}\left(\left(2k+1\right)^{2}\right)\label{eq:34-2}\\
\sigma_{n} & = & U_{n}\left(\left(2k+1\right)^{2}\right)+U_{n-1}\left(\left(2k+1\right)^{2}\right)\label{eq:35-2}
\end{eqnarray}

\noindent These results are generalized for all $k$ and $n\in\mathbb{Z}^{+}$
in the next 
\begin{thm}
\begin{flushleft}
$\forall k,n\in\mathbb{Z}^{+}$, $\exists\delta_{n},\sigma_{n},M_{n},a_{n},c_{n}\in\mathbb{Z}^{+}$,
$M_{n}>1$, such that (\ref{eq:4-1}) holds if $\sigma_{n}=\sqrt{\delta_{n}^{2}+\left(\delta_{n}^{2}-1\right)/2k\left(k+1\right)}$,
with
\begin{eqnarray}
M_{n} & = & k\delta_{n}\label{eq:36-1}\\
 & = & k\left[U_{n}\left(\left(2k+1\right)^{2}\right)-U_{n-1}\left(\left(2k+1\right)^{2}\right)\right]\label{eq:36-2}\\
a_{n} & = & \frac{\delta_{n}+1}{2}\label{eq:37-1}\\
 & = & \frac{U_{n}\left(\left(2k+1\right)^{2}\right)-U_{n-1}\left(\left(2k+1\right)^{2}\right)+1}{2}\label{eq:37-2}\\
c_{n} & = & \delta_{n}\sqrt{\frac{k\left(k+1\right)\left(\left(2k\left(k+1\right)+1\right)\delta_{n}^{2}-1\right)}{8}}\label{eq:38-1}\\
 & = & \frac{k\left(k+1\right)U_{2n}\left(\left(2k+1\right)^{2}\right)}{2}\label{eq:38-2}
\end{eqnarray}

\par\end{flushleft}\end{thm}
\begin{proof}
For $k,n,\delta_{n},\sigma_{n},g,M_{n},a_{n},c_{n}\in\mathbb{Z}^{+}$,
$M_{n}>1$, let 

\noindent $\sigma_{n}=\sqrt{\delta_{n}^{2}+\left(\delta_{n}^{2}-1\right)/2k\left(k+1\right)}$.
Replacing in (\ref{eq:14-1}) yields
\begin{eqnarray}
g & = & \frac{M_{n}}{2\delta_{n}^{4}}\left(\left(\frac{\left(2k\left(k+1\right)+1\right)\delta_{n}^{2}-1}{2k\left(k+1\right)}\right)M_{n}+\right.\nonumber \\
 &  & \left.\sqrt{M_{n}^{2}\left(\left(\frac{\left(2k\left(k+1\right)+1\right)\delta_{n}^{2}-1}{2k\left(k+1\right)}\right)^{2}-\delta_{n}^{4}\right)+\delta_{n}^{4}}\right)\label{eq:39-1}
\end{eqnarray}

\noindent For the polynomial in $\delta_{n}$ under the square root
sign to be a square, let $M_{n}=k\delta_{n}$, yielding after simplification
$g=k\left(k+1\right)/2=\triangle_{k}$. Replacing further in (\ref{eq:2-2})
and (\ref{eq:12-3}) yields (\ref{eq:37-1}) and (\ref{eq:38-1}).
Replacing further $\delta_{n}$ by (\ref{eq:34-2}) yields (\ref{eq:36-2})
and (\ref{eq:37-2}). Replacing $\delta_{n}$ and $\sigma_{n}$ in
(\ref{eq:12-3}) yields (\ref{eq:38-2}), noting that $\delta_{n}\sigma_{n}=U_{n}^{2}\left(\left(2k+1\right)^{2}\right)-U_{n-1}^{2}\left(\left(2k+1\right)^{2}\right)=U_{2n}\left(\left(2k+1\right)^{2}\right)$
as can be shown by replacing $U_{n}^{2}\left(x\right)$ and $U_{n-1}^{2}\left(x\right)$
in function of Chebyshev polynomials of the first kind, respectively
$T_{2n+2}\left(x\right)$ and $T_{2n}\left(x\right)$ (see e.g. \cite{key-1-1})
and simplifying appropriately.
\end{proof}
\noindent For the case (iii) of Theorem 2, the first $50\,000$ values
of $\left(M,a,c\right)$ are given in \cite{key-7-5}. There exist
also other triplets of values of $\left(M,a,c\right)$ such that (\ref{eq:4-1})
holds with $\delta=1$ and $\sigma=k$. Table 4 shows some of these
values for $1<k\leq5$ and $1<M_{n},a_{n}<10^{5}$ and indexed by
$n$ for increasing values of $M_{n}$.
\begin{table}
\noindent \begin{centering}
\protect\caption{Values of $M_{n},a_{n},g_{n},c_{n}$ with $\delta=1$ and $\sigma=k$
\protect \\
for $1<k\leq5$ and $1<M_{n},a_{n}<10^{5}$}

\par\end{centering}

\noindent \centering{}%
\begin{tabular}{|c|c|c|c|c|c|}
\hline 
$k$ & $n$ & $M_{n}$ & $a_{n}$ & $g_{n}$ & $c_{n}$\tabularnewline
\hline 
\hline 
2 & 1 & 8 & 28 & 252 & 504\tabularnewline
\hline 
2 & 2 & 63 & 217 & 15624 & 31248\tabularnewline
\hline 
2 & 3 & 496 & 1705 & 968440 & 1936880\tabularnewline
\hline 
2 & 4 & 3905 & 13420 & 60027660 & 120055320\tabularnewline
\hline 
\hline 
3 & 1 & 18 & 153 & 2907 & 8721\tabularnewline
\hline 
3 & 2 & 323 & 2737 & 936054 & 2808162\tabularnewline
\hline 
3 & 3 & 5796 & 49105 & 301406490 & 904219470\tabularnewline
\hline 
\hline 
4 & 1 & 32 & 496 & 16368 & 65472\tabularnewline
\hline 
4 & 2 & 1023 & 15841 & 16728096 & 66912384\tabularnewline
\hline 
\hline 
5 & 1 & 50 & 1225 & 62475 & 312375\tabularnewline
\hline 
5 & 2 & 2499 & 61201 & 156062550 & 780312750\tabularnewline
\hline 
\end{tabular}
\end{table}
 The generalization of the case (iii) of Theorem 2 to other values
of $M$ is included in a following more general theorem. This theorem
will use the solutions of simple and generalized Pell equations, that
are recalled in the next section.

\section{Pell equations: A Reminder}

\noindent Pell equations of the general form

\begin{equation}
X^{2}-DY^{2}=N\label{eq:1}
\end{equation}

\noindent with $X,Y,N\in\mathbb{Z}$ and square free $D\in\mathbb{Z}^{+}$,
i.e. $\sqrt{D}\notin\mathbb{Z}^{+}$, have been investigated in various
forms since long (see historical accounts in \cite{key-1-13,key-2,key-5,key-8})
and are treated in several classical text books (see e.g. \cite{key-7,key-3,key-1,key-1a}
and references therein). A simple reminder is given here and further
details can be found in the references.

\noindent For $N=1$, the simple Pell equation reads classically 
\begin{equation}
X^{2}-DY^{2}=1\label{eq:2}
\end{equation}
which has, beside the trivial solution $(X_{0},Y_{0})=(1,0)$, a whole
infinite branch of solutions $\forall n\in\mathbb{Z}^{+}$ given by
\begin{eqnarray}
X_{n} & = & \frac{\left(X_{1}+\sqrt{D}Y_{1}\right)^{n}+\left(X_{1}-\sqrt{D}Y_{1}\right)^{n}}{2}\label{eq:3}\\
Y_{n} & = & \frac{\left(X_{1}+\sqrt{D}Y_{1}\right)^{n}-\left(X_{1}-\sqrt{D}Y_{1}\right)^{n}}{2\sqrt{D}}\label{eq:4}
\end{eqnarray}
where $(X_{1},Y_{1})$ is the fundamental solution to (\ref{eq:2}),
i.e. the smallest integer solution different from the trivial solution
($X_{1}>1,Y_{1}>0,\in\mathbb{Z}^{+}$). Among the five methods listed
by Robertson \cite{key-9} to find the fundamental solution $(X_{1},Y_{1})$,
the classical method based on the continued fraction expansion of
the quadratic irrational $\sqrt{D}$ introduced by Lagrange \cite{key-10-1-1}
is at the core of several other methods. It can be summarized as follows.
One computes the $r^{\textnormal{th}}$ convergent $\left(p_{r}/q_{r}\right)$
of the continued fraction $\left[\alpha_{0};\alpha_{1},...,\alpha_{r},\alpha_{r+1},...\right]$
of $\sqrt{D}$ which becomes periodic after the following term $\alpha_{r+1}=2\alpha_{0}$
if $\sqrt{D}$ is a quadratic surd or quadratic irrational (i.e. $\sqrt{D}\notin\mathbb{Z}^{+}$)
and with $\alpha_{0}=\left\lfloor \sqrt{D}\right\rfloor $ i.e. the
greatest integer $\leq\sqrt{D}$ . The terms $p_{i}$ and $q_{i}$
can be found by the recurrence relations 
\begin{equation}
p_{i}=\alpha_{i}p_{i-1}+p_{i-2}\,\,,\,\,\, q_{i}=\alpha_{i}q_{i-1}+q_{i-2}\label{eq:5-1}
\end{equation}
with $p_{-2}=0,p_{-1}=1,q_{-2}=1,q_{-1}=0$. The fundamental solution
is then $(X_{1},Y_{1})=\left(p_{r},q_{r}\right)$ if $r\equiv1\left(mod\,2\right)$
or $(X_{1},Y_{1})=\left(p_{2r+1},q_{2r+1}\right)$ if $r\equiv0\left(mod\,2\right)$. 

\noindent For the general case of $N\neq1$, the generalized Pell
equation (\ref{eq:1}) can have either no solution at all, or one
or several fundamental solutions $\left(X_{1},Y_{1}\right)$, and
all integer solutions, if they exist, come on double infinite branches
that can be expressed in function of the fundamental solution(s) $\left(X_{1},Y_{1}\right)$
and $\left(X_{1},-Y_{1}\right)$. Several authors (see e.g. \cite{key-10-1-1,key-1-11,key-7,key-17,key-9,key-15,key-16,key-1a}
and references therein) discussed how to find the fundamental solution(s)
of the generalized Pell equation, based on Lagrange's method of continued
fractions with various modifications (see e.g. \cite{key-1-2}), and
further how to find additional solutions from the fundamental solution(s).
The method indicated by Matthews \cite{key-16} based on an algorithm
by Frattini \cite{key-6-1,key-6-2,key-6-3} using Nagell's bounds
\cite{key-7,key-6-4}, will be used further.

\noindent Noting now $\left(x_{f},y_{f}\right)$ the fundamental solutions
of the related simple Pell equation (\ref{eq:3}), the other solutions
$\left(X_{n},Y_{n}\right)$ can be found from the fundamental solution(s)
by 
\begin{equation}
X_{n}+\sqrt{D}Y_{n}=\pm\left(X_{1}+\sqrt{D}Y_{1}\right)\left(x_{f}+\sqrt{D}y_{f}\right)^{n}\label{eq:5}
\end{equation}
for a proper choice of sign $\pm$ \cite{key-9}, or by the recurrence
relations
\begin{eqnarray}
X_{n} & = & x_{f}X_{n-1}+Dy_{f}Y_{n-1}\label{eq:6}\\
Y_{n} & = & x_{f}Y_{n-1}+y_{f}X_{m-1}\label{eq:7}
\end{eqnarray}

\noindent that can also be written as
\begin{eqnarray}
X_{n} & = & 2x_{f}X_{n-1}-X_{n-2}\label{eq:8-1}\\
Y_{n} & = & 2x_{f}Y_{n-1}-Y_{n-2}\label{eq:9-1}
\end{eqnarray}

\noindent or in function of Chebyshev's polynomials of the first kind
$T_{n-1}\left(x_{f}\right)$ and of the second kind $U_{n-2}\left(x_{f}\right)$
evaluated at $x_{f}$ (see \cite{key-3-1})
\begin{eqnarray}
X_{n} & = & X_{1}T_{n-1}\left(x_{f}\right)+DY_{1}y_{f}U_{n-2}\left(x_{f}\right)\label{eq:8}\\
Y_{n} & = & X_{1}y_{f}U_{n-2}\left(x_{f}\right)+Y_{1}T_{n-1}\left(x_{f}\right)\label{eq:9}
\end{eqnarray}

\noindent For $N=\eta^{2}$ an integer square, the generalized Pell
equation (\ref{eq:1}) admits always integer solutions. The variable
change $\left(X^{\prime},Y^{\prime}\right)=\left(\left(X/\eta\right),\left(Y/\eta\right)\right)$
transforms the generalized Pell equation in a simple Pell equation
$X{}^{\prime2}-DY{}^{\prime2}=1$ which has integer solutions $\left(X_{n}^{\prime},Y_{n}^{\prime}\right)$.
The integer solutions to the generalized Pell equation can then be
found as $\left(X_{n},Y_{n}\right)=\left(\eta X_{n}^{\prime},\eta Y_{n}^{\prime}\right)$,
or from ((\ref{eq:8-1}),(\ref{eq:9-1})) with $\left(X_{0},Y_{0}\right)=\left(1,0\right)$
and $\left(X_{1},Y_{1}\right)=\left(\eta x_{f},\eta y_{f}\right)$,
yielding simply
\begin{eqnarray}
X_{n} & = & \eta T_{n}\left(x_{f}\right)\label{eq:12-2}\\
Y_{n} & = & \eta y_{f}U_{n-1}\left(x_{f}\right)\label{eq:13-1}
\end{eqnarray}

\noindent (which is also valid for the simple Pell equation (\ref{eq:2})
with $\eta=1$). Note however that not all solutions in $\left(X,Y\right)$
may be found in this way (see e.g.\cite{key-1}) and, depending on
the value of $D$, other fundamental solutions may exist.

\section{Step 3: A More General Approach}

\noindent The three cases of Theorem 2 can now be generalized as shown
in the next theorem, that includes also the general method to find
all solutions $\left(M,a,c\right)$ such that (\ref{eq:4-1}) holds.
\begin{thm}
\noindent For $\forall\delta,n\in\mathbb{Z}^{+}$, $\exists\sigma,a,M,M_{n},c,x_{f},y_{f},X_{1},Y_{1},D,N\in\mathbb{Z}^{+}$,
$\kappa\in\mathbb{Q}^{+}$, with $\gcd\left(\delta,\sigma\right)=1$,$\kappa=\left(\sigma/\delta\right)>1$,
such as (\ref{eq:4-1}) holds with 
\begin{eqnarray}
M_{n} & = & X_{1}y_{f}U_{n-1}\left(x_{f}\right)+Y_{1}T_{n}\left(x_{f}\right)\label{eq:4-3}\\
a_{n} & = & \frac{1}{2\delta^{2}}\left[\left(\sigma^{2}-\delta^{2}\right)\left(X_{1}+\left(\sigma^{2}+\delta^{2}\right)Y_{1}\right)y_{f}U_{n-1}\left(x_{f}\right)+\right.\label{eq:5-2-3}\\
 &  & \left.\left(X_{1}+\left(\sigma^{2}-\delta^{2}\right)Y_{1}\right)T_{n}\left(x_{f}\right)+\delta^{2}\right]\nonumber \\
c_{n} & = & \frac{\sigma}{2\delta^{3}}\left(X_{1}y_{f}U_{n-1}\left(x_{f}\right)+Y_{1}T_{n}\left(x_{f}\right)\right)\label{eq:5-3-1}\\
 &  & \left[\left(\sigma^{2}X_{1}+\left(\sigma^{4}-\delta^{4}\right)Y_{1}\right)y_{f}U_{n-1}\left(x_{f}\right)+\left(X_{1}+\sigma^{2}Y_{1}\right)T_{n}\left(x_{f}\right)\right]\nonumber 
\end{eqnarray}
where $\left(x_{f},y_{f}\right)$ and $\left(X_{1},Y_{1}\right)$
are the fundamental solutions of respectively the simple (\ref{eq:2})
and generalized Pell equations (\ref{eq:1}) with $D=\left(\sigma^{4}-\delta^{4}\right)$
and $N=\delta^{4}$, and $T_{n}\left(x_{f}\right)$ and $U_{n}\left(x_{f}\right)$
Chebyshev's Polynomials of the first and second kind evaluated at
$x_{f}$.\end{thm}
\begin{proof}
\noindent Let $n,M,M_{n},a,a_{n},c,c_{n},\sigma,\delta,x_{f},y_{f},X_{1},Y_{1}\in\mathbb{Z}^{+}$,
$k\in\mathbb{Z}^{*}$, $\kappa,C,C_{n}\in\mathbb{\mathbb{Q}}^{+}$,
with $M>1$, $\gcd\left(\delta,\sigma\right)=1$, $\kappa=\left(\sigma/\delta\right)>1$.
As $g$ (\ref{eq:14-1}), $a$ (\ref{eq:2-2}) and $c$ (\ref{eq:3-2})
in Theorem 1 must be integers, the condition for the polynomial$\left(M^{2}\left(\kappa^{4}-1\right)+1\right)$
under the square root sign in (\ref{eq:14-1}), (\ref{eq:2-2}) and
(\ref{eq:3-2}) to be a squared integer or a squared rational allows
to find for which values of $M$ (\ref{eq:4-1}) holds. Let 
\begin{equation}
M^{2}\left(\kappa^{4}-1\right)+1=C^{2}\label{eq:16-1}
\end{equation}

\noindent which can be rewritten as a Pell equation as 
\begin{equation}
C^{2}-\left(\kappa^{4}-1\right)M^{2}=1\label{eq:17-1}
\end{equation}

\noindent or as%
\footnote{\noindent Note that (\ref{eq:17-1}) could also be written as $C^{2}-\left(\sigma^{4}-\delta^{4}\right)\left(M/\delta^{2}\right)^{2}=1$,
but not all solutions may be obtained in this way.%
}
\begin{equation}
\left(\delta^{2}C\right)^{2}-\left(\sigma^{4}-\delta^{4}\right)M^{2}=\delta^{4}\label{eq:21}
\end{equation}

\noindent which is a generalized Pell equation that always admits
at least one fundamental solution as the right hand term is a squared
integer. Noting $\left(x_{f},y_{f}\right)$ the fundamental solutions
of the simple Pell equation (\ref{eq:2}) and $\left(X_{1},Y_{1}\right)$
the fundamental solution(s) of the generalized Pell equation (\ref{eq:1})
with $X=\left(\delta^{2}C\right)$, $Y=M$, $D=\left(\sigma^{4}-\delta^{4}\right)$
and $N=\delta^{4}$, all solutions can be found by (\ref{eq:5}) or
((\ref{eq:8}),(\ref{eq:9})) $\forall n\in\mathbb{Z}^{+}$ yielding
\begin{equation}
C_{n}=\frac{\left(\left(\sigma^{4}-\delta^{4}\right)Y_{1}y_{f}U_{n-1}\left(x_{f}\right)+X_{1}T_{n}\left(x_{f}\right)\right)}{\delta^{2}}\label{eq:5-7}
\end{equation}
and (\ref{eq:4-3}). Then replacing $M$ by $M_{n}$ (\ref{eq:4-3})
in (\ref{eq:2-2}) and (\ref{eq:3-2}) yields directly (\ref{eq:5-2-3})
and (\ref{eq:5-3-1}).
\end{proof}
\noindent Note that, although (\ref{eq:4-3}) yields all integer solutions
in $M$ to (\ref{eq:21}), some of them do not yield integer solutions
to $a_{n}$ (\ref{eq:5-2-3}) and $c_{n}$ (\ref{eq:5-3-1}) and must
be rejected.

\noindent For $\delta=1$, (\ref{eq:21}) is a simple Pell equation
similar to (\ref{eq:2}). It is easy to see that its fundamental solution
is $\left(x_{f},y_{f}\right)=\left(\sigma^{2},1\right)$ (see e.g.
\cite{key-1-2}). (\ref{eq:21}) admits then an infinitude of solutions
$\forall n\in\mathbb{Z}^{+}$ for each integer value of $\sigma$,
that can be written as ((\ref{eq:3}), (\ref{eq:4})) or ((\ref{eq:12-2}),
(\ref{eq:13-1})), yielding $C_{n}=T_{n}\left(\sigma^{2}\right)$
and
\begin{eqnarray}
M_{n} & = & U_{n-1}\left(\sigma^{2}\right)\label{eq:18}\\
a_{n} & = & \frac{U_{n}\left(\sigma^{2}\right)-U_{n-1}\left(\sigma^{2}\right)+1}{2}\label{eq:19}\\
c_{n} & = & \frac{\sigma U_{n}\left(\sigma^{2}\right)U_{n-1}\left(\sigma^{2}\right)}{2}\label{eq:20}
\end{eqnarray}

\noindent where the relation $U_{n}\left(\sigma^{2}\right)=T_{n}\left(\sigma^{2}\right)+\sigma^{2}U_{n-1}\left(\sigma^{2}\right)$
(see e.g. \cite{key-1-1}) was used in (\ref{eq:19}) and (\ref{eq:20}).
Note as well that $g_{n}=U_{n}\left(\sigma^{2}\right)U_{n-1}\left(\sigma^{2}\right)/2$
as can be found from (\ref{eq:12-3}) or (\ref{eq:12-1}).

\noindent This generalizes the case (iii) of Theorem 2 and gives an
infinitude of solutions $\left(M_{n},a_{n},c_{n}\right)$ $\forall n\in\mathbb{Z}^{+}$
for $\delta=1$ and for each value of $\sigma$, $\forall\sigma\in\mathbb{Z}^{+}$.

\noindent For $\delta>1$, three of the fundamental solutions are
always $\left(X_{1},Y_{1}\right)=\left(\delta^{2},0\right)$ and $\left(\sigma^{2},\pm1\right)$,
corresponding to respectively $\left(C,M\right)=\left(1,0\right)$
and $\left(\left(\sigma^{2}/\delta^{2}\right),\pm1\right)$. Depending
on the value of $D=\left(\sigma^{4}-\delta^{4}\right)$, other fundamental
solutions may exist. All solutions in $M$ can then be found by (\ref{eq:4-3})
$\forall n\in\mathbb{Z}^{+}$ for $\left(X_{1},Y_{1}\right)=\left(\delta^{2},0\right)$
or $\left(\sigma^{2},\pm1\right)$, and $\forall n\in\mathbb{Z}^{*}$
for other fundamental solutions. Furthermore, solutions found for
$\left(X_{1},Y_{1}\right)=\left(\sigma^{2},-1\right)$ and $\left(\delta^{2},0\right)$
yield integer values for $M_{n},a_{n}$ and $c_{n}$, while the solutions
found for $\left(X_{1},Y_{1}\right)=\left(\sigma^{2},1\right)$, although
yielding integer values of $M_{n}$, do not yield integer values for
$a_{n}$ and $c_{n}$, as can be seen easily from (\ref{eq:5-2-3})
and (\ref{eq:5-3-1}), and must be rejected, although these non-integer
values satisfy (\ref{eq:4-1}). 

\noindent Relations (\ref{eq:5-7}) and (\ref{eq:4-3}) to (\ref{eq:5-3-1})
read

\noindent - for $\left(X_{1},Y_{1}\right)=\left(\sigma^{2},-1\right)$,
$C_{n}=\kappa^{2}T_{n}\left(x_{f}\right)-\delta^{2}\left(\kappa^{4}-1\right)y_{f}U_{n-1}\left(x_{f}\right)$
and 
\begin{eqnarray}
M_{n} & = & \sigma^{2}y_{f}U_{n-1}\left(x_{f}\right)-T_{n}\left(x_{f}\right)\label{eq:18-1-1}\\
a_{n} & = & \frac{T_{n}\left(x_{f}\right)-\left(\sigma^{2}-\delta^{2}\right)y_{f}U_{n-1}\left(x_{f}\right)+1}{2}\label{eq:19-1-1}\\
c_{n} & = & \frac{\sigma\delta y_{f}U_{n-1}\left(x_{f}\right)\left(\sigma^{2}y_{f}U_{n-1}\left(x_{f}\right)-T_{n}\left(x_{f}\right)\right)}{2}\label{eq:20-1-1}
\end{eqnarray}

\noindent - for $\left(X_{1},Y_{1}\right)=\left(\delta^{2},0\right)$,
$C_{n}=T_{n}\left(x_{f}\right)$ and 
\begin{eqnarray}
M_{n} & = & \delta^{2}y_{f}U_{n-1}\left(x_{f}\right)\label{eq:18-1}\\
a_{n} & = & \frac{T_{n}\left(x_{f}\right)+\left(\sigma^{2}-\delta^{2}\right)y_{f}U_{n-1}\left(x_{f}\right)+1}{2}\label{eq:19-1}\\
c_{n} & = & \frac{\sigma\delta y_{f}U_{n-1}\left(x_{f}\right)\left(\sigma^{2}y_{f}U_{n-1}\left(x_{f}\right)+T_{n}\left(x_{f}\right)\right)}{2}\label{eq:20-1}
\end{eqnarray}

\noindent Note that the solutions (\ref{eq:18-1-1}) to (\ref{eq:20-1-1})
found for $\left(X_{1},Y_{1}\right)=\left(\sigma^{2},-1\right)$ are
smaller than those (\ref{eq:18-1}) to (\ref{eq:20-1}) found for
$\left(X_{1},Y_{1}\right)=\left(\delta^{2},0\right)$ for a same value
of $n\geq1$. 

\noindent Furthermore, if $2g$ is a multiple of $M,$ $2g=\mu M$,
$\mu\in\mathbb{Z}^{+}$, (like for $M=33,35,42...)$, the fundamental
solutions of the simple Pell equation are $\left(x_{f},y_{f}\right)=\left(\left(\mu\sigma^{2}-M\right),\mu\right)$,
as can be easily verified.

\noindent If $g=\triangle_{M}$ (like for $M=5,15,...$) or $\triangle_{M-1}$
(like for $M=5,13,15,17,40,...$), then
\begin{eqnarray}
\left(x_{f},y_{f}\right) & = & \left(\left(\left(M\pm1\right)\sigma^{2}-M\right),\left(M\pm1\right)\right)\label{eq:17}\\
 & = & \left(\left(2\left|\delta^{4}-2\sigma^{2}+1\right|\left(\frac{\sigma^{2}-1}{\left|\delta^{4}-2\sigma^{2}+1\right|}\right)^{2}-1\right),\right.\nonumber \\
 &  & \left.\left(\frac{2\left(\sigma^{2}-1\right)}{\left|\delta^{4}-2\sigma^{2}+1\right|}\right)\right)\label{eq:18-2}
\end{eqnarray}

\noindent with $M=\left(\delta^{4}-1\right)/\left(\left|2\sigma^{2}-\delta^{4}-1\right|\right)$
(see \cite{key-5-6}) with the $+$ (respectively $-$) sign in (\ref{eq:17})
for $g=\triangle_{M}$ (resp. $\triangle_{M-1}$) and vertical bars
denote the absolute value. 

\noindent Values of $\left(x_{f},y_{f}\right)$ and $\left(X_{1},Y_{1}\right)$
from \cite{key-16} yielding solutions $\left(M_{n},a_{n},c_{n}\right)$
for the first values of $M$, $1<M\leq45$ and $1<a<10^{5}$ of Table
1 can be found in \cite{key-5-6}. 

\noindent For the case (i) of Theorem 2 with $M=\left(2k+1\right)$,
$g=\left(M^{2}-1\right)/4$, $\delta=2M$ and $\sigma=\left(2M^{2}-1\right)$,
one of the other fundamental solutions is always
\begin{eqnarray}
\left(X_{1},Y_{1}\right) & = & \left(\left(\frac{\delta\left(\sigma^{2}-2\right)}{2}\right),\frac{\delta}{2}\right)\label{eq:5-14}\\
 & = & \left(M\left(\left(2M^{2}-1\right)^{2}-2\right),M\right)\label{eq:5-15}
\end{eqnarray}
as can be easily verified. More generally, all solutions for this
case can be found by (\ref{eq:5-2}) to (\ref{eq:6-3}).

\noindent For the case (ii) of Theorem 2 with $a=\left(2k+1\right)^{2}$,
$g=\triangle_{k}$, $\delta=16\triangle_{k}+1=2a-1$, $\sigma=16\triangle_{k}+3=2a+1$,
$M=k\delta=k\left(16\triangle_{k}+1\right)$, one of the other fundamental
solutions is always known and can be expressed in function of $\delta$,
$k$ or $a$ as
\begin{eqnarray}
\left(X_{1},Y_{1}\right) & = & \left(\left(4\delta\left(\sqrt{\frac{\delta+1}{2}}\right)^{3}\left(\sqrt{\frac{\delta+1}{2}}-1\right)\right),2\delta\right)\\
 & = & \left(\left(4\left(16\triangle_{k}+1\right)\left(8k\left(2k+1\right)^{3}+1\right)\right),2\left(16\triangle_{k}+1\right)\right)\\
 & = & \left(\left(2a-1\right)a^{2}\left(1-a^{-1/2}\right),2\left(2a-1\right)\right)
\end{eqnarray}

\section{Recurrent Relations}

\noindent The sets of solutions $\left(M_{n},a_{n},c_{n}\right)$
are obviously not independent. As (\ref{eq:4-3}) to (\ref{eq:5-3-1})
are linear combinations of Chebyshev polynomials, one has also the
general recurrence relations 

\noindent 
\begin{eqnarray}
M_{n} & = & 2x_{f}M_{n-1}-M_{n-2}\label{eq:6-1-3-1}\\
a_{n} & = & 2x_{f}a_{n-1}-a_{n-2}-\left(x_{f}-1\right)\label{eq:6-2-1}\\
c_{n} & = & 2\left(2x_{f}^{2}-1\right)c_{n-1}-c_{n-2}+\delta\sigma^{3}y_{f}^{2}\label{eq:6-3-1-1}
\end{eqnarray}

\noindent among values of $M_{n},a_{n},c_{n}$ calculated for the
same fundamental solutions $\left(X_{1},Y_{1}\right)$. These relations
are immediate from (\ref{eq:4-3}) to (\ref{eq:5-3-1}) and the recurrence
and other formulas for Chebyshev polynomials (see e.g \cite{key-1-1}).
For the sake of the recurrence, initial values of $\left(M_{n},a_{n},c_{n}\right)$
for $n=0,1$ are shown in Table 5 for the cases (i) $M_{n}=k\delta_{n}$
of Theorem 3, (ii) $\delta=1$ and $\left(x_{f},y_{f}\right)=\left(\sigma^{2},1\right)$,
(iii) $\delta>1$ and $\left(X_{1},Y_{1}\right)=\left(\sigma^{2},-1\right)$,
(iv) $\delta>1$ and $\left(X_{1},Y_{1}\right)=\left(\delta^{2},0\right)$,
and (v) the general case for $\delta>1$ and with other fundamental
solutions $\left(X_{1}^{*},Y_{1}^{*}\right)$ of the generalized Pell
equation.
\begin{table}
\protect\caption{Initial values of $\left(M_{n},a_{n},c_{n}\right)$ for $n=0,1$ for
recurrence relations}

\noindent \centering{}%
\begin{tabular}{|c|c|l|}
\hline 
{\footnotesize{}$X_{1},Y_{1}$} & {\footnotesize{}$n$} & {\footnotesize{}$\left(M_{n},a_{n},c_{n}\right)$ for $M_{n}=k\delta_{n}$,
$\left(x_{f},y_{f}\right)=\left(\left(2k+1\right)^{2},0\right)$}\tabularnewline
\hline 
{\footnotesize{}$\left[1,0\right]$} & {\footnotesize{}$0$} & {\footnotesize{}$\left(1,1,1\right)$}\tabularnewline
\cline{2-3} 
 & {\footnotesize{}$1$} & {\footnotesize{}$\left(k\left(8k\left(k+1\right)+1\right),\left(2k+1\right)^{2},\left(k\left(k+1\right)/2\right)\left(4\left(2k+1\right)^{4}-1\right)\right)$}\tabularnewline
\hline 
\hline 
{\footnotesize{}$X_{1},Y_{1}$} & {\footnotesize{}$n$} & {\footnotesize{}$\left(M_{n},a_{n},c_{n}\right)$ for $\delta=1$,
$\left(x_{f},y_{f}\right)=\left(\sigma^{2},1\right)$}\tabularnewline
\hline 
{\footnotesize{}$\left[1,0\right]$} & {\footnotesize{}$0$} & {\footnotesize{}$\left(0,1,0\right)$}\tabularnewline
\cline{2-3} 
 & {\footnotesize{}$1$} & {\footnotesize{}$\left(1,\sigma^{2},\sigma^{3}\right)$}\tabularnewline
\hline 
\hline 
{\footnotesize{}$X_{1},Y_{1}$} & {\footnotesize{}$n$} & {\footnotesize{}$\left(M_{n},a_{n},c_{n}\right)$ for $\delta>1$}\tabularnewline
\hline 
{\footnotesize{}$\sigma^{2},-1$} & {\footnotesize{}$0$} & {\footnotesize{}$\left(-1,1,0\right)$}\tabularnewline
\cline{2-3} 
 & {\footnotesize{}$1$} & {\footnotesize{}$\left(\left(\sigma^{2}y_{f}-x_{f}\right),\left(x_{f}-\left(\sigma^{2}-\delta^{2}\right)y_{f}+1\right)/2,\,\,\delta\sigma y_{f}\left(\sigma^{2}y_{f}-x_{f}\right)/2\right)$}\tabularnewline
\hline 
{\footnotesize{}$\delta^{2},0$} & {\footnotesize{}$0$} & {\footnotesize{}$\left(0,1,0\right)$}\tabularnewline
\cline{2-3} 
 & {\footnotesize{}$1$} & {\footnotesize{}$\left(\delta^{2}y_{f},\,\,\left(x_{f}+\left(\sigma^{2}-\delta^{2}\right)y_{f}+1\right)/2,\,\,\delta\sigma y_{f}\left(\sigma^{2}y_{f}+x_{f}\right)/2\right)$}\tabularnewline
\hline 
{\footnotesize{}$X_{1}^{*},Y_{1}^{*}$} & {\footnotesize{}$0$} & {\footnotesize{}$\left(Y_{1}^{*},\left(X_{1}^{*}+\left(\sigma^{2}-\delta^{2}\right)Y_{1}^{*}+\delta^{2}\right)/2\delta^{2},\,\,\sigma Y_{1}^{*}\left(X_{1}^{*}+\sigma^{2}Y_{1}^{*}\right)/2\delta^{3}\right)$}\tabularnewline
\cline{2-3} 
 & {\footnotesize{}$1$} & {\footnotesize{}$\left(X_{1}^{*}y_{f}+Y_{1}^{*}x_{f},\right.$}\tabularnewline
 &  & {\footnotesize{}$\left(\left(x_{f}+\left(\sigma^{2}-\delta^{2}\right)y_{f}\right)X_{1}^{*}+\left(\sigma^{2}-\delta^{2}\right)\left(x_{f}+\left(\sigma^{2}+\delta^{2}\right)y_{f}\right)Y_{1}^{*}+\delta^{2}\right)/2\delta^{2},$}\tabularnewline
 &  & {\footnotesize{}$\left.\sigma\left(y_{f}X_{1}^{*}+x_{f}Y_{1}^{*}\right)\left(\left(x_{f}+\sigma^{2}y_{f}\right)X_{1}^{*}+\left(\sigma^{2}x_{f}+\left(\sigma^{4}-\delta^{4}\right)y_{f}\right)Y_{1}^{*}\right)/2\delta^{3}\right)$}\tabularnewline
\hline 
\end{tabular}
\end{table}

\noindent For $\delta=1$, one has also the remarkable recurrence
relation 
\begin{equation}
M_{n}=M_{n-1}+2a_{n-1}-1\label{eq:6-4-1}
\end{equation}

\noindent among all values of $M_{n}$ $\forall n\in\mathbb{Z}^{+}$
(see Table 7 further in section 7). One has also a similar relation
among values of $M_{n}$ and $a_{n}$ for $\delta>1$ if only those
solutions (\ref{eq:18-1-1}), (\ref{eq:19-1-1}) and (\ref{eq:18-1}),
(\ref{eq:19-1}) calculated respectively for $\left(X_{1},Y_{1}\right)=\left(\sigma^{2},-1\right)$
and $\left(\delta^{2},0\right)$ are considered. 

\noindent If all solutions$\left(M_{n},a_{n},c_{n}\right)$ are ordered
in increasing value order and indexed by a new index $j\in\mathbb{Z}^{+}$,
one obtains simply $j=n$ for solutions for $\delta=1$, and for $\delta>1$,
if there are no fundamental solutions other than $\left(X_{1},Y_{1}\right)=\left(\delta^{2},0\right)$
and $\left(\sigma^{2},\pm1\right)$, one obtains $j=2n-1$ for the
solutions (\ref{eq:18-1-1}) to (\ref{eq:20-1-1}) and $j=2n$ for
the solutions (\ref{eq:18-1}) to (\ref{eq:20-1}). If fundamental
solutions other than $\left(X_{1},Y_{1}\right)=\left(\delta^{2},0\right)$
and $\left(\sigma^{2},\pm1\right)$ exist, the solutions $\left(M_{n},a_{n},c_{n}\right)$
(\ref{eq:4-3}) to (\ref{eq:5-3-1}) calculated with these other fundamental
solutions (including for $n=0)$ have to be inserted accordingly (see
example further in Section 7).

\noindent For $\delta>1$, if only those value of $\left(M_{j},a_{j},c_{j}\right)$
(\ref{eq:18-1-1}) and \ref{eq:20-1-1}) and $\left(M_{j+1},a_{j+1},c_{j+1}\right)$
(\ref{eq:18-1} to \ref{eq:20-1}) calculated respectively with $\left(X_{1},Y_{1}\right)=\left(\sigma^{2},-1\right)$
and $\left(\delta^{2},0\right)$ are considered, two such sets of
solutions are called a ``recurrent pair'' and obviously, for $\delta=1$,
all sets of $\left(M_{j},a_{j},c_{j}\right)$ solutions form ``recurrent
pairs''. The following theorem give other remarkable recurrent relations.
\begin{thm}
For $\forall\delta\in\mathbb{Z}^{+}$, $\exists\sigma,\kappa,j,a_{j},M_{j},c_{j}\in\mathbb{Z}^{+}$
with $\gcd\left(\delta,\sigma\right)=1$, $\kappa=\left(\sigma/\delta\right)>1$,
and such as (\ref{eq:4-1}) holds, if $\left(M_{j},a_{j},c_{j}\right)$
and $\left(M_{j+1},a_{j+1},c_{j+1}\right)$ form a ``recurrent pair'',
then 
\begin{eqnarray}
M_{j+1} & = & M_{j}+2a_{j}-1\label{eq:25-1}\\
 & = & \kappa^{2}M_{j}+\sqrt{M_{j}^{2}\left(\kappa^{4}-1\right)+1}=\kappa^{2}M_{j}+C_{j}\label{eq:25-2}\\
M_{j} & = & \left(2\kappa^{2}-1\right)M_{j+1}-2a_{j+1}+1\label{eq:26}\\
a_{j+1}+a_{j} & = & \kappa^{2}M_{j+1}-M_{j}+1\label{eq:27-1}\\
a_{j+1}-a_{j} & = & \left(\kappa^{2}-1\right)M_{j+1}\label{eq:28}\\
c_{j} & = & \frac{\kappa M_{j}M_{j+1}}{2}\label{eq:29-1}\\
c_{j+1}+c_{j} & = & \kappa^{3}M_{j+1}^{2}\label{eq:30}
\end{eqnarray}
\end{thm}
\begin{proof}
Let $\sigma,\delta,j,n,k,a,M,c\in\mathbb{Z}^{+}$ with $\gcd\left(\delta,\sigma\right)=1$,
$\kappa=\left(\sigma/\delta\right)>1$, and $\left(M_{n},a_{n},c_{n}\right)$
solutions of (\ref{eq:4-1}), yielding $\left(M_{j},a_{j},c_{j}\right)$
and $\left(M_{j+1},a_{j+1},c_{j+1}\right)$ to be a ``recurrent pair''
with $n=j$ for $\delta=1$, and, for $\delta>1$, $n=j$ or $n=j+1$
respectively for $\left(M_{n},a_{n},c_{n}\right)$ (\ref{eq:18-1-1})
to (\ref{eq:20-1-1}) for $\left(X_{1},Y_{1}\right)=\left(\sigma^{2},-1\right)$
or (\ref{eq:18-1}) to (\ref{eq:20-1}) for $\left(X_{1},Y_{1}\right)=\left(\delta^{2},0\right)$.
Then,

\noindent (i) (\ref{eq:25-1}) is immediate from (\ref{eq:18}) and
(\ref{eq:19}) for $\delta=1$, and from (\ref{eq:18-1-1}), (\ref{eq:19-1-1})
and (\ref{eq:18-1}) for $\delta>1$.

\noindent (ii) (\ref{eq:25-2}) is immediate from (\ref{eq:12-1}),
(\ref{eq:14-1}) and (\ref{eq:16-1}). 

\noindent (iii) For $\delta=1$, replacing $M_{j+2}$ by the recurrence
relation (\ref{eq:6-1-3-1}) in (\ref{eq:25-1}) written for $M_{j+2}$
yields directly (\ref{eq:26}). For $\delta>1$, in $\left(2\kappa^{2}-1\right)M_{j+1}-2a_{j+1}+1$,
replacing $M_{j+1}$ and $a_{j+1}$ by (\ref{eq:18-1}) and (\ref{eq:19-1})
yields
\begin{eqnarray}
\left(2\kappa^{2}-1\right)M_{j+1}-2a_{j+1}+1 & = & \left(2\sigma^{2}-\delta^{2}\right)y_{f}U_{n-1}\left(x_{f}\right)-\nonumber \\
 &  & \left(T_{n}\left(x_{f}\right)+\left(\sigma^{2}-\delta^{2}\right)y_{f}U_{n-1}\left(x_{f}\right)\right)\label{eq:32-1}\\
 & = & \sigma^{2}y_{f}U_{n-1}\left(x_{f}\right)-T_{n}\left(x_{f}\right)\label{eq:33-1}
\end{eqnarray}

\noindent which is $M_{n}$ (\ref{eq:18-1-1}). Therefore (\ref{eq:26})
holds for $\delta\geq1$.

\noindent (iv) Summing and subtracting $a_{j+1}$ and $a_{j}$ extracted
respectively from (\ref{eq:26}) and (\ref{eq:25-1}) yield directly
(\ref{eq:27-1}) and (\ref{eq:28}). 

\noindent (v) (\ref{eq:3-2}) and (\ref{eq:25-2}) yield directly
(\ref{eq:29-1}). 

\noindent (vi) Replacing $M_{j}$ from (\ref{eq:26}) in (\ref{eq:29-1})
and replacing with $c_{j+1}$ from (\ref{eq:29-1}) yield directly
(\ref{eq:30}).
\end{proof}
\noindent This theorem means that once a solution $\left(M_{n},a_{n},c_{n}\right)$
has been found for $\delta>1$ from (\ref{eq:18-1-1}) to (\ref{eq:20-1-1})
with $\left(X_{1},Y_{1}\right)=\left(\sigma^{2},-1\right)$, the other
solution $\left(M_{n},a_{n},c_{n}\right)$ of the \textquotedbl{}recurrent
pair\textquotedbl{} can be found directly from (\ref{eq:25-1}) to
(\ref{eq:30}) without having to be calculated from (\ref{eq:18-1})
to (\ref{eq:20-1}) for $\left(X_{1},Y_{1}\right)=\left(\delta^{2},0\right)$.

\section{Summary and Examples}

\noindent In summary, to find all solutions $\left(M,a,c\right)$
such that the sum of $M$ consecutive cubed integers starting from
$a>1$ equal a squared integer $c^{2}$, the following approach is
proposed.

\noindent 1) Calculate first all solutions for odd integer values
of $M=\left(2k+1\right)$ $\forall k\in\mathbb{Z}^{+}$ by (\ref{eq:5-2})
to (\ref{eq:6-3}). Table 6 shows the first ten values of an infinitude
of solutions. 
\begin{table}
\protect\caption{Values of $\left(M,a,c\right)$ for $M=\left(2k+1\right)$ , $1\leq k\leq10$}

\centering{}%
\begin{tabular}{|l|}
\hline 
{\footnotesize{}(3, 23, 204), (5, 118, 2940), (7, 333, 16296), (9,
716, 57960), (11, 1315, 159060), }\tabularnewline
{\footnotesize{}(13, 2178, 368004), (15, 3353, 754320), (17, 4888,
1412496), (19, 6831, 2465820),}\tabularnewline
{\footnotesize{}(21, 9230, 4070220)}\tabularnewline
\hline 
\end{tabular}
\end{table}

\noindent 2) Second, calculate all solutions for the cases $M_{n}=k\delta_{n}$
given by Theorem 3 $\forall k,n\in\mathbb{Z}^{+}$ either by (\ref{eq:36-2}),
(\ref{eq:37-2}), (\ref{eq:38-2}) in function of Chebyshev polynomials,
or by the recurrence relations (\ref{eq:6-1-3-1}) to (\ref{eq:6-3-1-1})
with $\left(x_{f},y_{f}\right)=\left(\left(2k+1\right)^{2},0\right)$
and the initial values of Table 5 (see Table 3).

\noindent 3) Third, calculate all solutions for the case $\delta=1$
$\forall\sigma\in\mathbb{Z}^{+}$ either by (\ref{eq:18}) to (\ref{eq:20}),
or by the recurrence relations (\ref{eq:6-1-3-1}) to (\ref{eq:6-3-1-1})
with $\left(x_{f},y_{f}\right)=\left(\sigma^{2},1\right)$ and the
initial values of Table 5. Table 7 shows the first five values of
the infinitude of solutions for $2\leq\sigma\leq5$. 
\begin{table}
\protect\caption{Values of $\left(M_{n},a_{n},c_{n}\right)$ for $\delta=1$, $2\leq\sigma\leq5$
and $2\leq n\leq6$}

\centering{}%
\begin{tabular}{|l|l|}
\hline 
\multicolumn{1}{|c|}{$\sigma=2$} & \multicolumn{1}{c|}{$\sigma=3$}\tabularnewline
\hline 
{\footnotesize{}(8, 28, 504)} & {\footnotesize{}(18, 153, 8721)}\tabularnewline
{\footnotesize{}(63, 217, 31248)} & {\footnotesize{}(323, 2737, 2808162)}\tabularnewline
{\footnotesize{}(496, 1705, 1936880)} & {\footnotesize{}(5796, 49105, 904219470)}\tabularnewline
{\footnotesize{}(3905, 13420, 120055320)} & {\footnotesize{}(104005, 881145, 291155861205)}\tabularnewline
{\footnotesize{}(30744, 105652, 7441492968)} & {\footnotesize{}(1866294, 15811497, 93751283088567)}\tabularnewline
\hline 
\multicolumn{1}{|c|}{$\sigma=4$} & \multicolumn{1}{c|}{$\sigma=5$}\tabularnewline
\hline 
{\footnotesize{}(32, 496, 65472)} & {\footnotesize{}(50, 1225, 312375)}\tabularnewline
{\footnotesize{}(1023, 15841, 66912384)} & {\footnotesize{}(2499, 61201, 780312750)}\tabularnewline
{\footnotesize{}(32704, 506401, 68384391040)} & {\footnotesize{}(124900, 3058801, 1949220937250)}\tabularnewline
{\footnotesize{}(1045505, 16188976, 69888780730560)} & {\footnotesize{}(6242501, 152878825, 4869153120937875)}\tabularnewline
{\footnotesize{}(33423456, 517540816, 71426265522241344)} & {\footnotesize{}(312000150, 7640882425, 12163142546881874625)}\tabularnewline
\hline 
\end{tabular}
\end{table}

\noindent 4) Fourth, for all other cases with $\delta>1$, find the
fundamental solutions of the simple and generalized Pell equations
with $D=\left(\sigma^{4}-\delta^{4}\right)$ and $N=\delta^{4}$ for
all values of $\sigma\in\mathbb{Z}^{+}$ and of $\delta\in\mathbb{Z}^{+}$
such that $1<\delta<\sigma$ and $\gcd\left(\delta,\sigma\right)=1$.

\noindent 4.1) If there are no other fundamental solutions than $\left(X_{1},Y_{1}\right)=\left(\delta^{2},0\right)$
and $\left(\sigma^{2},\pm1\right)$, then $g=\mu M/2$ ($\mu\in\mathbb{Z}^{+}$)
and $\left(x_{f},y_{f}\right)=\left(\left(\mu\sigma^{2}-M\right),\mu\right)$;
if $g=\triangle_{M}$ or $\triangle_{M-1}$, then $M_{1}=\left(\delta^{4}-1\right)/\left(\left|\delta^{4}-2\sigma^{2}+1\right|\right)$
and (\ref{eq:18-2}) gives $\left(x_{f},y_{f}\right)$. Then (\ref{eq:18-1-1})
to (\ref{eq:20-1}) yield an infinitude of integer solutions $\left(M_{n},a_{n},c_{n}\right)$
$\forall n\in\mathbb{Z}^{+}$ with $\left(X_{1},Y_{1}\right)=\left(\sigma^{2},-1\right),\left(\delta^{2},0\right)$.
Alternatively, the recurrence relations (\ref{eq:6-1-3-1}) to (\ref{eq:6-3-1-1})
with the initial values of Table 5 or the recurrence relations of
Theorem 5 can be used. Solutions have to be ordered then in increasing
order of $M_{j}$.

\noindent For $\sigma=3$ and $\delta=2$, $g=120=\triangle_{15}$,
$D=65$ and $N=16$ yielding $\left(x_{f},y_{f}\right)=\left(129,16\right)$
and only three fundamental solutions $\left(X_{1},Y_{1}\right)=\left(4,0\right)$
and $\left(9,\pm1\right)$ (see \cite{key-16}). The first ten solutions
are shown in Table 8 for $\left(X_{1},Y_{1}\right)=\left(9,-1\right)$
and $\left(4,0\right)$, arranged by $M_{j}$ increasing values and
with respectively $j=2n-1$ and $j=2n$. 
\begin{table}
\protect\caption{Values of $\left(M_{j},a_{j},c_{j}\right)$ for $\delta=2$, $\sigma=3$
with $\left(x_{f},y_{f}\right)=\left(129,16\right)$ and $1\leq n\leq5$
with $j=2n-1$ and $j=2n$ respectively for $\left(X_{1},Y_{1}\right)=\left(9,-1\right)$
and $\left(4,0\right)$}

\noindent \centering{}%
\begin{tabular}{|c|c|c|l|}
\hline 
{\footnotesize{}$X_{1},Y_{1}$} & {\footnotesize{}$n$} & {\footnotesize{}$j$} & {\footnotesize{}$\left(M_{j},a_{j},c_{j}\right)$}\tabularnewline
\hline 
\hline 
{\footnotesize{}9, -1} & {\footnotesize{}1} & {\footnotesize{}1} & {\footnotesize{}(15, 25, 720)}\tabularnewline
\hline 
{\footnotesize{}4, 0} & {\footnotesize{}1} & {\footnotesize{}2} & {\footnotesize{}(64, 105, 13104)}\tabularnewline
\hline 
{\footnotesize{}9, -1} & {\footnotesize{}2} & {\footnotesize{}3} & {\footnotesize{}(3871, 6321, 47938464)}\tabularnewline
\hline 
{\footnotesize{}4, 0} & {\footnotesize{}2} & {\footnotesize{}4} & {\footnotesize{}(16512, 26961, 872242272)}\tabularnewline
\hline 
{\footnotesize{}9, -1} & {\footnotesize{}3} & {\footnotesize{}5} & {\footnotesize{}(998703, 1630665, 3190880053872)}\tabularnewline
\hline 
{\footnotesize{}4, 0} & {\footnotesize{}3} & {\footnotesize{}6} & {\footnotesize{}(4260032, 6955705, 58058190109584)}\tabularnewline
\hline 
{\footnotesize{}9, -1} & {\footnotesize{}4} & {\footnotesize{}7} & {\footnotesize{}(257661503, 420705121, 212391358097903424)}\tabularnewline
\hline 
{\footnotesize{}4, 0} & {\footnotesize{}4} & {\footnotesize{}8} & {\footnotesize{}(1099071744, 1794544801, 3864469249201901760)}\tabularnewline
\hline 
{\footnotesize{}9, -1} & {\footnotesize{}5} & {\footnotesize{}9} & {\footnotesize{}(66475669071, 108540290425, 14137193574521767668240)}\tabularnewline
\hline 
{\footnotesize{}4, 0} & {\footnotesize{}5} & {\footnotesize{}10} & {\footnotesize{}(283556249920, 462985602825, 257226802107318794853360)}\tabularnewline
\hline 
\end{tabular}
\end{table}

\noindent 4.2) If there are fundamental solutions $\left(X_{1},Y_{1}\right)=\left(X_{1}^{*},\pm Y_{1}^{*}\right)$
other than $\left(\delta^{2},0\right)$ and $\left(\sigma^{2},\pm1\right)$,
then $\left(x_{f},y_{f}\right)$ has to be calculated separately (see
\cite{key-16}). All integer solutions $\left(M_{n},a_{n},c_{n}\right)$
are found by (\ref{eq:4-3}) to (\ref{eq:5-3-1}) for $\left(X_{1},Y_{1}\right)=\left(X_{1}^{*},Y_{1}^{*}\right)$
$\forall n\in\mathbb{Z}^{*}$ and by (\ref{eq:18-1-1}) to (\ref{eq:20-1})
for $\left(X_{1},Y_{1}\right)=\left(\sigma^{2},-1\right),\left(\delta^{2},0\right)$
$\forall n\in\mathbb{Z}^{+}$. Alternatively, the recurrence relations
(\ref{eq:6-1-3-1}) to (\ref{eq:6-3-1-1}) can be used with the initial
values of Table 5. Solutions have to be ordered then in increasing
order of $M_{j}$. 

\noindent For $\sigma=4$ and $\delta=3$, $D=175$ and $N=81$ yielding
$\left(x_{f},y_{f}\right)=\left(2024,153\right)$ and five fundamental
solutions $\left(X_{1},Y_{1}\right)=\left(9,0\right),\left(16,\pm1\right)$
and $\left(159,\pm12\right)$ (see \cite{key-16}). The first ten
solutions are shown in Table 9 for $\left(X_{1},Y_{1}\right)=\left(159,12\right),\left(16,-1\right)$
and $\left(9,0\right)$, arranged by $M_{j}$ increasing values and
with respectively $j=3n+1$, $j=3n-1$ and $j=3n$. 
\begin{table}
\protect\caption{Values of $\left(M_{j},a_{j},c_{j}\right)$ for $\delta=3$, $\sigma=4$
with $\left(x_{f},y_{f}\right)=\left(2024,153\right)$ and $0\leq n\leq3$
with $j=3n+1$, $j=3n-1$ and $j=3n$ respectively for $\left(X_{1},Y_{1}\right)=\left(159,12\right),\left(16,-1\right)$
and $\left(9,0\right)$ }

\noindent \centering{}%
\begin{tabular}{|c|c|c|l|}
\hline 
{\footnotesize{}$X_{1},Y_{1}$} & {\footnotesize{}$n$} & {\footnotesize{}$j$} & {\footnotesize{}$\left(M_{j},a_{j},c_{j}\right)$ }\tabularnewline
\hline 
\hline 
{\footnotesize{}159,12} & {\footnotesize{}0} & {\footnotesize{}1} & {\footnotesize{}(12, 14, 312)}\tabularnewline
\hline 
{\footnotesize{}16,-1} & {\footnotesize{}1} & {\footnotesize{}2} & {\footnotesize{}(424, 477, 389232)}\tabularnewline
\hline 
{\footnotesize{}9,0} & {\footnotesize{}1} & {\footnotesize{}3} & {\footnotesize{}(1377, 1548, 4105296)}\tabularnewline
\hline 
{\footnotesize{}159,12} & {\footnotesize{}1} & {\footnotesize{}4} & {\footnotesize{}(48615, 54635, 5117020440)}\tabularnewline
\hline 
{\footnotesize{}16,-1} & {\footnotesize{}2} & {\footnotesize{}5} & {\footnotesize{}(1716353, 1928872, 6378077594592)}\tabularnewline
\hline 
{\footnotesize{}9,0} & {\footnotesize{}2} & {\footnotesize{}6} & {\footnotesize{}(5574096, 6264280, 67270624549920)}\tabularnewline
\hline 
{\footnotesize{}159,12} & {\footnotesize{}2} & {\footnotesize{}7} & {\footnotesize{}(196793508, 221160443, 83849042274507096)}\tabularnewline
\hline 
{\footnotesize{}16,-1} & {\footnotesize{}3} & {\footnotesize{}8} & {\footnotesize{}(6947796520, 7808071356, 104513105644422184080)}\tabularnewline
\hline 
{\footnotesize{}9,0} & {\footnotesize{}3} & {\footnotesize{}9} & {\footnotesize{}(22563939231, 25357801869, 1102316769603603585072)}\tabularnewline
\hline 
{\footnotesize{}159,12} & {\footnotesize{}3} & {\footnotesize{}10} & {\footnotesize{}(796620071769, 895257416606, 1373975729120835063673080)}\tabularnewline
\hline 
\end{tabular}
\end{table}

\section{\noindent Conclusions}

\noindent The approach proposed in this paper to find all solutions
$\left(M,a,c\right)$ to (\ref{eq:2-3}) by investigating two other
parameters, $\delta$ and $\sigma$, instead of investigating each
individual values of $M$ one by one in an increasing order of $M$
values allows to find quite simple and elegant general solutions $\left(M,a,c\right)$
based on solutions of simple and generalized Pell equations involving
Chebyshev polynomials. Alternatively, recurrence relations can be
used in order to simplify the computational part. This approach allows
to find all possible solutions $\left(M,a,c\right)$ to (\ref{eq:2-3})
but with the drawback that solutions are not ordered by increasing
$M$ values. 

\noindent However, it is found that there are always at least one
solution for every cases of all odd values of $M$, of all odd integer
square values of $a$, and of all even values of $M$ equal to twice
an integer square.


\begin{thebibliography}{10}
\bibitem[1]{key-1-41} L. Aubry, Sphinx-Oedipe, 8, 28-9, 1913.

\bibitem[2]{key-1-47} M. Cantor, Nouvelles Annales de Mathématiques,
2, 6, 276-278, 1867.

\bibitem[3]{key-1-43} J.W.S. Cassels, A Diophantine equation, Glasgow
Mathematical Journal, 27, 11-18, 1985.

\bibitem[4]{key-1-45} E. Catalan, Bulletin Académie Royale de Belgique,
2, 22, 339-340, 1866.

\bibitem[5]{key-1-46} E. Catalan, Nouvelles Annales de Mathématiques,
2, 6, 63-67, 1867.

\bibitem[6]{key-1-11} G. Chrystal, Algebra - An Elementary Text-Book,
Part II, 1st ed. Adam and Charles Black, 1900; 2nd ed., New York,
Chelsea, 478-488, 1961. 

\bibitem[7]{key-1-13} L.E. Dickson, Sum of cubes of numbers in arithmetical
progression a square, Ch. XXI in History of the Theory of Numbers,
Vol. 2: Diophantine Analysis, Dover, New York, 585-588, 2005.

\bibitem[8]{key-6-1} G. Frattini, Dell\textquoteright analisi indeterminata
di secondo grado, Periodico di Mat. VI, 169\textendash 180, 1891.

\bibitem[9]{key-6-2} G. Frattini, A complemento di alcuni teore mi
del sig. Tchebicheff, Rom. Acc. L. Rend. 5, I No. 2, 85-91, 1892.

\bibitem[10]{key-6-3} G. Frattini, Dell\textquoteright analisi indeterminata
di secondo grado, Periodico di Mat. VII, 172\textendash 177, 1892. 

\bibitem[11]{key-1a} M. Jacobson, H. Williams, Solving the Pell Equation,
Springer-Verlag New York, 2009.

\bibitem[12]{key-10-1-1} J.L. Lagrange, Solution d'un Problème d'Arithmétique,
in Oeuvres de Lagrange, J.-A. Serret (ed.), Vol. 1, Gauthier-Villars,
Paris, 671\textendash 731, 1867 (see http://gdz.sub.uni-goettingen.de/en/dms/loader/img/?PPN=PPN308899644\&DMDID=
DMDLOG\_0024\&LOGID=LOG\_0024\&PHYSID=PHYS\_0726, last accessed 10
August 2014).

\bibitem[13]{key-8} F. Lemmermeyer, Pell equation bibliography, 1658-1943,
http://www.fen.bilkent.edu.tr/\textasciitilde{}franz/publ/pell.html,
last accessed 2 November 2013.

\bibitem[14]{key-2} H. W. Lenstra Jr., Solving the Pell Equation,
Notices of The AMS, Vol. 49, Nr 2, 182-192, 2002.

\bibitem[15]{key-1-40} E. Lucas, Recherches sur l'analyse indéterminée,
Moulins, 92, 1873. Extract from Bull. Soc d'Emulation du Département
de l'Allier, 12, 532, 1873.

\bibitem[16]{key-15} K.R. Matthews, The Diophantine Equation $x^{2}-Dy^{2}=N$,
$D>0$, in integers, Expositiones Mathematicae, 18, 323-331, 2000.

\bibitem[17]{key-16} K. Matthews, Quadratic Diophantine equations
BCMATH programs, http://www.numbertheory.org/php/main\_pell.html,
last accessed 3 January 2015.

\bibitem[18]{key-6-4} K. Matthews, J. Robertson, On the converse
of a theorem of Nagell and Tchebicheff, Preprint submitted to Expositiones
Mathematicae, 6 April 2014 (see http://www.numbertheory.org/pdfs/nagell2.pdf,
last accessed 27 July 2014).

\bibitem[19]{key-17} R.A. Mollin, Fundamental Number Theory with
Applications, CRC Press, New York, 294-307, 1998.

\bibitem[20]{key-7} T. Nagell, Introduction to Number Theory, Wiley,
New York, 195-212, 1951.

\bibitem[21]{key-3} J.J. O'Connor and E.F. Robertson, Pell's equation,
JOC/EFR February 2002 http://www-history.mcs.st-andrews.ac.uk/HistTopics/Pell.html,
last accessed 26 July 2014.

\bibitem[22]{key-1-2} V. Pletser, On continued fraction development
of quadratic irrationals having all periodic terms but last equal
and associated general solutions of the Pell equation, Journal of
Number Theory, Vol. 136, 339\textendash 353, 2013.

\bibitem[23]{key-3-1} V. Pletser, Finding all squared integers expressible
as the sum of consecutive squared integers using generalized Pell
equation solutions with Chebyshev polynomials, ArXiv, http://arxiv.org/abs/1409.7972,
29 September 2014.

\bibitem[24]{key-5-2} V. Pletser, Numbers whose square equals sums
of cubes of consecutive integers, Sequence A180921 in The On-line
Encyclopedia of Integer Sequences, published electronically at http://oeis.org,
24 September 2010, last accessed 3 January 2015. 

\bibitem[25]{key-5-3} V. Pletser, Numbers of terms in sums of cubes
of consecutive integers equal to squared integers, Sequence A180920
in The On-line Encyclopedia of Integer Sequences, published electronically
at http://oeis.org, 24 September 2010, last accessed 3 January 2015.

\bibitem[26]{key-5-5} V. Pletser, Number of terms, first term and
square root of sums of consecutive cubed integers equal to integer
squares, ResearchGate, https://www.researchgate.net/publication/271272786\_Number\_of\_terms
\_first\_term\_and\_square\_root\_of\_sums\_of\_consecutive\_cubed\_integers
\_equal\_to\_integer\_squares, 24 January 2015.

\bibitem[27]{key-7-0} V. Pletser, Triplets (n, x, y) with n,x less
than 10\textasciicircum{}5, https://oeis.org/A218979/a218979\_1.txt,
10 January 2015, in M. Markus, Numbers n such that the sum of the
n consecutive positive cubes is a square, Sequence A218979 in The
On-line Encyclopedia of Integer Sequences, published electronically
at http://oeis.org, last accessed 24 January 2015.

\bibitem[28]{key-7-1} V. Pletser, Numbers a(n) that are the starting
terms in the sum of an odd number of consecutive cubes equal to a
square, Sequence A253679 in The On-line Encyclopedia of Integer Sequences,
published electronically at http://oeis.org and https://oeis.org/A253679/a253679.txt,
8 January 2015.

\bibitem[29]{key-7-2} V. Pletser, Numbers c(n) whose square are equal
to the sum of an odd number M of consecutive cubed integers b\textasciicircum{}3
+ (b+1)\textasciicircum{}3 + ... + (b+M-1)\textasciicircum{}3 = c\textasciicircum{}2,
starting at b(n) (A253679)., Sequence A253680 in The On-line Encyclopedia
of Integer Sequences, published electronically at http://oeis.org,
submitted 8 January 2015.

\bibitem[30]{key-7-3} V. Pletser, Numbers M(n) which are the number
of terms in the sums of consecutive cubed integers equaling a squared
integer, b\textasciicircum{}3 + (b+1)\textasciicircum{}3 + ... + (b+M-1)\textasciicircum{}3
= c\textasciicircum{}2, for a first term b(n) being an odd squared
integer (A016754)., Sequence A253707 in The On-line Encyclopedia of
Integer Sequences, published electronically at http://oeis.org, submitted
9 January 2015.

\bibitem[31]{key-7-4} V. Pletser, Numbers c(n) whose squares are
equal to the sums of consecutive cubed integers b\textasciicircum{}3
+ (b+1)\textasciicircum{}3 + ... + (b+M-1)\textasciicircum{}3 = c\textasciicircum{}2,
for a first term b(n) being an odd squared integer (A016754)., Sequence
A253708 in The On-line Encyclopedia of Integer Sequences, published
electronically at http://oeis.org, submitted 9 January 2015.

\bibitem[32]{key-7-5} V. Pletser, Numbers c(n) whose squares are
equal to the sums of a number M(n) of consecutive cubed integers b\textasciicircum{}3
+ (b+1)\textasciicircum{}3 + ... + (b+M-1)\textasciicircum{}3 = c\textasciicircum{}2,
starting at b(n) (A002593) for M(n) being twice a squared integer
(A001105)., Sequence A253724 in The On-line Encyclopedia of Integer
Sequences, published electronically at http://oeis.org, submitted
10 January 2015.

\bibitem[33]{key-5-6} V. Pletser, Fundamental solutions of the Pell
equation $X^{2}-\left(\sigma^{4}-\delta^{4}\right)Y^{2}=\delta^{4}$
for the first 45 solutions of the sums of consecutive cubed integers
equalling integer squares, ArXiv, http://arxiv.org/abs/1167238, submitted
24 January 2015. 

\bibitem[34]{key-1-48} C. Richaud, Zeitschrift für Mathematik und
Physik, 12, 170-172, 1867.

\bibitem[35]{key-9} J.P. Robertson, Solving the generalized Pell
equation $X^{2}-DY^{2}=N$, 31 July 2004 (see http://www.jpr2718.org/pell.pdf,
last accessed 26 July 2014).

\bibitem[36]{key-1-1} J. Spanier and K.B. Oldham, An Atlas of Functions,
Springer-Verlag, 193-207, 1987.

\bibitem[37]{key-1-44} R.J. Stroeker, On the sum of consecutive cubes
being a perfect square, Compositio Mathematica, 97, 295-307, 1995.

\bibitem[38]{key-5} A. Weil, Number Theory, an Approach through History,
Birkhäuser, Boston, 1984.

\bibitem[39]{key-1} E.W. Weisstein, Pell Equation, from MathWorld--A
Wolfram Web Resource. http://mathworld.wolfram.com/PellEquation.html,
last accessed 10 August 2014.\end{thebibliography}
\end{document}